\documentclass[11pt, twoside]{article}
\usepackage{amsmath,amsthm,amssymb}
\usepackage{times}
\usepackage{enumerate}  

\usepackage{color}
\usepackage[colorlinks]{hyperref} 

\def\titlerunning#1{\gdef\titrun{#1}}
\makeatletter
\def\author#1{\gdef\autrun{\def\and{\unskip, }#1}\gdef\@author{#1}}
\def\address#1{{\def\and{\\\hspace*{18pt}}\renewcommand{\thefootnote}{}%
\footnote {#1}}%
\markboth{\autrun}{\titrun}}
\makeatother
\def\email#1{e-mail: #1}
\def\subjclass#1{{\renewcommand{\thefootnote}{}%
\footnote{\emph{Mathematics Subject Classification (2010):} #1}}}
\def\keywords#1{\par\medskip
\noindent\textbf{Keywords.} #1}

\newtheorem{thm}{Theorem}[section]
\newtheorem{cor}[thm]{Corollary}
\newtheorem{lemma}[thm]{Lemma}

\newtheorem{remark}[thm]{Remark}

\theoremstyle{definition}

\numberwithin{equation}{section}

\def\N{\mathbb {N}}

\def\R{\mathbb {R}}

\def\C{\mathbb {C}}

\def\B{\mathcal {B}}

\def\D{\mathbb {D}}

\def\eps{\varepsilon}
\def\ie{{\em i.e.,\ }}

\begin{document}


\baselineskip=17pt


\titlerunning{Upper and lower bounds for the correlation function via general return times}
\title{Upper and lower bounds for the correlation function via inducing with general return times}

\author{Henk Bruin and Dalia Terhesiu}

\date{\today}

\maketitle

\address{Faculty of Mathematics, University of Vienna, Oskar Morgensternplatz 1, 1090 Vienna, Austria; \email{henk.bruin@univie.ac.at}
\and
Faculty of Mathematics, University of Vienna, Oskar Morgensternplatz 1, 1090 Vienna, Austria; \email{daliaterhesiu@gmail.com}}



\begin{abstract}
For non-uniformly expanding maps inducing with a general return
time to Gibbs Markov maps, we provide sufficient
 conditions for obtaining higher order asymptotics for the correlation
function in the infinite measure setting. 
Along the way, we show that these conditions are sufficient
to recover previous results on sharp mixing rates in the finite measure
setting for non-Markov maps, but for a larger class of observables.
The results are illustrated by (finite and infinite measure preserving)
non-Markov intervals maps with an indifferent fixed point.
\end{abstract}




\section{Introduction}

For infinite measure preserving systems $(X, f, \mu)$, first order asymptotics of the correlation function $\rho_n(v,w) = \int_X v w \circ f^n \, d\mu$
for suitable observables $v, w$ were obtained in ~\cite{Gouezel11, MT} via the method of operator renewal theory.
The method and techniques in~\cite{Gouezel11, MT}  rely on the existence of some $Y\subset X$ such that the
\emph{first return} map $f^\tau: Y\to Y$ (with {\em first return time} $\tau(y)=\inf\{n\ge1:f^ny\in Y\}$) 
satisfies several good properties. In particular, given that $\mu_\tau$ is  the $f^\tau$-invariant measure, it is essential that $\mu_\tau(y\in Y:\tau(y)>n)=\ell(n)n^{-\beta}$, where  $\ell$ is a slowly varying
function\footnote{We recall that a measurable
function $\ell:(0,\infty)\to(0,\infty)$ is slowly varying if $\lim_{x\to\infty}\ell(\lambda x)/\ell(x)=1$ for all $\lambda>0$. Examples of slowly varying functions are the asymptotically constant functions and the logarithm.} and
$\beta\in(0,1]$. The strong requirement on the asymptotic behavior of $\mu_\tau(\tau>n)$ originates in the works~\cite{GL, Erickson} which provide first order asymptotics for scalar renewal sequences.

As clarified in~\cite{MT}, in the infinite measure setting, higher order asymptotics of the correlation function $\rho_n(v,w)$ (for suitable observables $v, w$) can be obtained 
exploiting higher order expansion of the tail probability $\mu_\tau(\tau>n)$, 
which can be estimated when the invariant density of the first return map $f^\tau$ is smooth enough. The work~\cite{MT} obtains results on  higher order asymptotics of $\rho_n(v,w)$
associated with infinite measure systems that induce with \emph{first return} to Gibbs Markov maps (see Section~\ref{sec-towers} for
details); in particular, the results in~\cite{MT} (and later~\cite{Terhesiu12} which improves on~\cite{MT}) apply to the LSV family of maps considered in~\cite{LiveraniSaussolVaienti99}.
A return map with smooth invariant density may be obtained by inducing with a 
\emph{general return}, rather than a \emph{first return} time.
More precisely, even when the first return map is not Gibbs Markov, it might happen that there exists some {\em general} return time $\varphi:Y\to\N$ of $f$ to $Y$, 
such that $f^\varphi = (f^\tau)^\rho$ is Gibbs Markov
for some  {\em reinduce time} $\rho:Y \to \N$.
 A good example is the class of non-Markov maps with indifferent fixed points studied in~\cite{Zweimuller98,Zweimuller00}. 
 Higher order asymptotics of the correlation function $\rho_n$ for the infinite measure preserving systems studied in~\cite{Zweimuller98,Zweimuller00} has not been addressed yet.
It is mainly
the question of  higher order asymptotics of $\rho_n$ for
\emph{infinite} measure preserving systems that induce with a \emph{general return} time to Gibbs Markov maps that we answer in this paper. 
While focusing on this problem we also obtain some new results for \emph{finite} measure preserving systems  that induce with,  again,
a \emph{general return} time to Gibbs Markov maps. The method of proof builds 
on~\cite{MT} (in the infinite case) and on the works~\cite{Sarig02, Gouezel04} 
(in the finite case),
which develop operator renewal theory (via first return inducing) for dynamical systems.

Let $(X,\mu)$ be a measure space (finite or infinite), and 
$f:X\to X$ a conservative, ergodic 
measure preserving map. Fix $Y\subset X$ with $\mu(Y)\in(0,\infty)$ and let $\tau$ be the first return of $f$ to $Y$.
Let $L:L^1(\mu)\to L^1(\mu)$ denote the transfer operator 
 for $f$ and define
\[
T_n v:=1_YL^n (1_Yv) ,\enspace n\ge0, \qquad R_nv:=1_YL^n(1_{\{\tau=n\}}v),\enspace n\ge1.
\]
Thus $T_n$ corresponds to general returns to $Y$ and $R_n$ corresponds to first returns to $Y$. The  sequence of operators $T_n=\sum_{j=1}^n R_jT_{n-j}$
generalizes the notion of scalar renewal sequences (for details on the latter we refer to~\cite{Feller66, BGT} and references therein).

Operator renewal sequences via inducing with respect to the \emph{first return} time were introduced in~\cite{Sarig02} to study
lower bounds for 
the correlation function $\rho_n(v,w)$ (for $v,w$ supported on $Y$)
associated with finite measure preserving systems.
This technique was later refined in~\cite{Gouezel04, GouezelPhD}.
In particular, under suitable assumptions on the
first return map $f^\tau$, preserving a measure $\mu_\tau$, and requiring that
$\mu_\tau(y\in Y:\tau(y)>n)=O(n^{-\beta})$, $\beta>2$, ~\cite[Theorem 1]{Sarig02} provides higher order expansions of $T_n$,
while ~\cite[Theorem 1]{Gouezel04} shows that~\cite[Theorem 1]{Sarig02}
holds for $\beta>1$ (see also Subsection~\ref{subsec-prev} where we recall the latter mentioned result in a particular setting).
An immediate consequence of these results is that  the upper bound
$|\int_X v w\circ f^n\, d\mu -\int_X vd\mu\int_X w\, d\mu| = O(n^{-(\beta-1)})$
(for appropriate observables $v, w$ supported on $Y$) is sharp in the sense that there exists a lower bounds of the same order.

The work~\cite{MT} developed a theory of renewal operator sequences for dynamical systems with
infinite measure, generalizing the results of~\cite{GL, Erickson} to the operator case.
This work obtains first and higher order asymptotics for the $n$-th iterate 
 $L^n$ of the transfer operator associated with $f$. In particular, 
under suitable assumptions on the
first return map $f^\tau$ (including the assumption that $\mu_\tau(\tau>n)=\ell(n)n^{-\beta}$, where  $\ell$ is a slowly varying
function and $\mu_\tau$ is $f^\tau$ invariant) it is shown in~\cite{MT} that for $\beta\in (1/2,1)$,
$\lim_{n\to\infty}\ell(n)n^{1-\beta}L^n v=\frac{\sin\pi\beta}{\pi} \int v\, d\mu$, uniformly on $Y$ and pointwise on $X$, for appropriate
 observables $v$. Obviously, this type of result implies that $\lim_{n\to\infty}\ell(n)n^{1-\beta}\rho_n(v,w)=\frac{\sin\pi\beta}{\pi} \int v\, d\mu\,\int w\, d\mu $, for suitable observables $v,w$
 and $\beta\in (1/2,1)$.
For results for $\beta\leq 1/2$ under stronger tail assumptions we refer to~\cite{Gouezel11}.

An important question is whether operator renewal type results/arguments can be 
exploited for $(Y, F=f^\varphi)$, when $\varphi$ is a {\em general} return time 
of $f$ to $Y$ in the sense that
$f^\varphi = (f^\tau)^\rho$, where $\rho:Y \to \N$ is some {\em reinduce time}. 
Assume that $(Y, f^\varphi)$ is Gibbs Markov  preserving a
 measure $\mu_0$.
A Young tower over $f^\varphi$ can be constructed (see~\cite{Young99})
and the first return map on the base of the tower is isomorphic to $(Y, f^\varphi, \mu_0)$ (see Section~\ref{sec-towers}, which recalls
this in detail).

In this work we provide sufficient conditions to answer the above question when the general return map $f^\varphi$
is Gibbs Markov.
In short, we formulate a tail condition on $\mu_0(\rho > k)$
that allows us to work with a decomposition (as in \cite{Gouezel05,GouezelPhD}) of the transfer operator on the Young
tower over $(Y, \mu_0, f^\varphi)$.
Our main result in the finite case 
Theorem~\ref{prop-corel-finite} 
provides upper and lower bounds for the correlation function
$\rho_n(v,w)$  provided that the tails
$\mu_0(\varphi > n)$ and $\mu_0(\rho > k)$ are of the right form
(see (H0) a) and (H1) in Section~\ref{sec-frame}).
More importantly, Theorem~\ref{prop-corel-finite} provides upper and lower bounds of the correlation function
for observables $v,w$ supported on the
whole space $X$ (so not just on $Y$). To deal with observables supported on the whole space, we introduce
 weighted norms, with weights inverse proportional 
to the entrance time to $Y$ (see Section~\ref{sec-abstr}).

Our main result in the infinite case Theorem~\ref{prop-corel-infinite} provides higher order asymptotics of the correlation function
for observables supported on the whole space. This result is obtained assuming higher order expansion of $\mu_0(\varphi > n)$
(see (H0) b) in Section~\ref{sec-frame}) and again, assuming that the tail $\mu_0(\rho > k)$ satisfies (H1) in Section~\ref{sec-frame}.
To deal with observables supported on the whole space, we use the same type of  weighted norms used in the finite case (see Section~\ref{sec-abstr}).

We illustrate the use of  the main results 
in the setting of non-Markov interval maps with indifferent fixed points, in particular the class of maps studied
in \cite{Zweimuller98,Zweimuller00} (see
Section~\ref{sec-AFN}). Below, we recall briefly  the  previous results on the correlation function in this non-Markov setting.

 In the finite measure non-Markov setting (as in \cite{Zweimuller98,Zweimuller00}),
upper bounds for observables supported on the whole space $X$
have been obtained in \cite{MT13}. Although not written up yet, the method in~\cite{MT13}
can also be used to obtain lower bounds for the decay of correlations of observables supported on $Y$
and, most probably, can be extended to deal with observables supported on the whole space.
We also mention that in the same setting, the work~\cite{HuVaienti} provides upper and lower bounds for the decay of correlations of observables supported on $Y$.
In both works~\cite{MT13, HuVaienti}  the results are obtained
by exploiting operator renewal type results/arguments as developed
in \cite{Sarig02, Gouezel04} via inducing with first return times.

In the infinite setting of non-Markov maps $f:X\to X$ (as in \cite{Zweimuller98,Zweimuller00}), first order asymptotics of $L^nv$,
for some appropriate $v$ supported on a compact subset of $X':=X\setminus\{\text{indifferent fixed points}\}$
has been established in~\cite{MT}. This result immediately implies first order asymptotics of $\rho_n(v,w)$,
again for $v,w$ supported on $X'$. Again, the underlying scheme relies on 
inducing with first return times.
The detailed results are recalled in Section~\ref{sec-knownres}.

In the setting of finite measure preserving non-Markov interval maps with indifferent fixed points,
Theorems~\ref{prop-corel-finite} gives upper and lower bounds for the decay of correlation  of observables supported on $X$.
In the infinite measure setting, Theorem~\ref{prop-corel-infinite} gives  higher order asymptotics of $\rho_n(v,w)$ for $v,w$ supported  on $X$.
In checking  the required assumptions of the abstract results (\ie (H0) and (H1)) for typical examples in the class considered in \cite{Zweimuller98,Zweimuller00},
we obtain an excellent estimate on $ \mu_\tau(\tau > n)$. In the infinite measure case, the estimate on $\mu_\tau(\tau > n)$ enables us to extend the results of~\cite{MT,Terhesiu12} 
on the higher order asymptotics of $L^n$ to the typical examples studied here; we refer to Section~\ref{sec-moregen} for details.
\\[2mm]
{\bf Notation:} We will use $a_n = O(b_n)$ and (to make proofs more readable) also  $a_n \ll b_n$  to
mean that there is a uniform constant $C$ such that $a_n \leq C b_n$.

\section{The induced map and main assumptions}
\label{sec-frame}

Given $f:X\to X$, we require that there exists $Y\subset X$ and a \emph{general} (not necessarily first) return time $\varphi: Y\to \N$ such that
the return map $F:=f^\varphi:Y \to Y$, preserving the measure $\mu_0$, is a Gibbs Markov map as recalled below.
For convenience we rescale such that $\mu_0(Y)=1$.

We assume that $F$ has a Markov partition $\alpha = \{ a \}$ such that
$\varphi|_a$ is constant on each partition element, and
$F:a \to Y$ is a bijection $\bmod{\, \mu_0}$. 
Let $p = \log \frac{d\mu}{d\mu\circ F}$ be the corresponding potential.
We assume that there is $\theta \in (0,1)$ and 
$C_1 > 0$ such that
\begin{equation}\label{eq:locLip}
e^{p(y)} \leq C_1 \mu_0(a),\qquad |e^{p(y)}  - e^{p(y')}| \leq C_1 \mu(a) \theta^{s(y, y')} 
\quad \text{ for all } y, y' \in a,\,  a \in \alpha,
\end{equation}
where $s(y_1, y_2) = \min\{ n \geq 0 : F^ny_1\text{ and } F^ny_2 \text{ belong to different elements of } \alpha \}$
is the {\em separation time}. We also assume that $\inf_{a \in \alpha}\mu_0(Fa)>0$ (big image property).

In addition to the Gibbs Markov property above, we assume
that the following holds:

\begin{itemize}
\item[\textbf{(H0)}]
\begin{itemize}
\item[\textbf{a)}] Finite measure case: $\mu_0(y \in  Y:\varphi(y)>n)=O(n^{-\beta})$
for $\beta>1$.
\item[\textbf{b)}] Infinite measure case:
$\mu_0(y\in Y:\varphi(y)>n)=cn^{-\beta}+H(n)$  for $\beta\in (1/2,1)$,
some $c > 0$ and
function $H$ such that  $H(n)=O(n^{-2\beta})$.
\end{itemize}
\end{itemize}

The following dynamical assumption will be verified for the class of maps described 
in Section~\ref{sec-AFN} and will play an important role in the proofs of the main results.

\begin{itemize}
\item[\textbf{(H1)}] Let $\tau:Y\to\N$ be the first return time to $Y$,
and $\tau_k$ the
$k$-th return time to $Y$, \ie $\tau_0 = 0$,
$\tau_{k+1}(y) = \tau_k(y) + \tau(f^{\tau_k(y)}(y))$.
Let $\rho$ be the {\em reinduce time} for
the general return, \ie $\varphi(y) = \tau_{\rho(y)}(y)$.
Write $\{ \varphi > n\}: = \{ y \in Y : \varphi(y) > n\}$.
We assume that there exists $C > 0$ such that
\[
\int_{\{\varphi > n\}} \rho(y) \, d\mu_0 \leq C \mu_0(\varphi > n)
\]
for all $n \geq 0$.
\end{itemize}

\begin{remark}\label{rmk-expansion}
The first return time $\tau$ may be defined on a
larger set than where the general return time
$\varphi$ is defined, but the difference in domains has
measure zero, so we will ignore it.
\end{remark}

In order to have the norms in \eqref{eq-extracond2} below well defined,
we need another mild condition on the inducing scheme.
\begin{itemize}
\item[\textbf{(H2)}]
Either $f^i(a) \subset Y$ or
$f^i(a) \cap Y = \emptyset$ for all $a \in \alpha, 0 \leq i < \varphi(a)$, 
\end{itemize}
This condition certainly holds 
for the examples studied in Section~\ref{sec-AFN} (which provides the required details).

\section{The tower over the map $F=f^\varphi$}
\label{sec-towers}

The tower $\Delta$ is the disjoint union 
of sets $(\{ \varphi = j\} ,i)$, $j \geq 1$, $0 \leq i < j$
with tower map
$$
T_\Delta(y,i) = \begin{cases}
(y,i+1) & \text{ if } 0 \leq i < \varphi(y)-1,\\
(Fy,0) & \text{ if } i = \varphi(y)-1.
\end{cases}
$$
This map preserves the measure $\mu_\Delta$ defined as
$\mu_\Delta(A,i) = \mu_0(A)$ for every measurable set $A$, with $A \subset \{ \varphi = j\}$ and $0 \leq i < j$.

Let $Y_i = \{ (y,i) : \varphi(y) > i\}$ be the $i$-th level of the tower,
so $Y   = Y_0$ is the base.
The restriction $\mu_\Delta|_{Y} = \mu_0$ is invariant under 
$T_\Delta^\varphi$, which is the first return map to the base.

We extend the function $\varphi$ to the tower as
\begin{equation}\label{eq-genDelta}
\varphi_\Delta(y,i) := \varphi(y) - i.
\end{equation}
Define $\pi:\Delta \mapsto X$ by $\pi(y,i) := f^i(y)$.
Then $\mu_X = \mu_\Delta \circ \pi^{-1}$ is $f$-invariant, and $\mu_X$ is related 
to the $F$-invariant measure $\mu_0$ by the usual formula
$$
\mu_X(A) := \sum_{j=1}^\infty \sum_{i=0}^{j-1} \mu_0(f^{-i} \cap \{ \varphi = j\} )
= \sum_{j=0}^\infty \mu_0(f^{-j}A \cap \{ \varphi > j \} ).
$$
Regardless of whether $\bar\varphi := \int_Y \varphi\, d \mu_0$
is finite (in which case we can normalize $\mu_X$) or not, $\mu_0$
is absolutely continuous w.r.t.\ $\mu_X$.
 
Let $v_X, w_X$ be observables supported on the original space $X$;
they lift to observables on the tower which we will
denote by $v_\Delta = v_X \circ \pi$ and $w_\Delta = w_X \circ \pi$.
Then
\begin{equation}\label{eq:muDeltamuX}
 \int_X v_X w_X \circ f^n \, d\mu_X = 
\int_\Delta v_\Delta w_\Delta \circ T_\Delta^n \, d\mu_\Delta. 
\end{equation}
To justify \eqref{eq:muDeltamuX}, use the duality formula
$\int_\Delta \pi^*v_X w_\Delta d\mu_\Delta = \int_X v_X \hat \pi  w_\Delta d\mu_X$,
where $\pi^*v_X = v_X \circ \pi$ and $\hat \pi w_\Delta = w_\Delta \circ \pi^{-1}$.
To compute $ \int_X v_X w_X \circ f^n \, d\mu_X$, it therefore suffices to estimate
the correlation function on the tower.

\section{Results for the map $f$ under the assumptions of Section~\ref{sec-frame}}
\label{sec-abstr}

Throughout we assume that $f$ and $F=f^\varphi$ satisfy the assumptions of Section~\ref{sec-frame}. In particular, we assume that $F$
is Gibbs Markov and that the relevant forms of 
(H0) and (H1) hold.

We restrict to the following class of observables. Let
$$
\tau^*(x) := 1+\min\{ i \geq 0 : f^i(x) \in Y\}.
$$

Recall that $s(x, x')$
is the separation time of points $x,x'\in Y$ and let $\theta \in (0,1)$ be such that ~\eqref{eq:locLip} holds.
Let $v_X:X\to\R$. For $\eps > 0$ we define the weighted norm $\|\, \|_\theta^*$ as follows: 
\begin{equation}~\label{eq-extracond2}
\begin{cases}
\|v_X\|_{\infty}^* := \sup_{x\in X}|v_X(x)|\tau^*(x)^{1+\epsilon}, \\[2mm]
|v_X |_\theta^* = \sup_{a \in \alpha} \sup_{0 \leq i < \varphi(a)} \sup_{x,x' \in a}
 \frac{(\tau^* \circ f^i(a))^{1+\eps}}{\theta^{s(x,x')}} |v_X \circ f^i(x) - v_X \circ f^i(x')|,
\end{cases}
\end{equation}
and $\| v_X \|_\theta^* = \| v_X \|_\infty^* + | v_X |_\theta^*$. 
Note that by (H2),  $\tau^*$ is constant on $f^i(a)$, 
so the factor $\tau^* \circ f^i(a)$ in \eqref{eq-extracond2} 
is well-defined.
\begin{remark}\label{rmk-suppY1}
If $v_X$ is supported on $Y$, then the weighted norms $\|\,\|_{\infty}^*$ and $\| \, \|_\theta^*$ coincide with 
$\|\,\|_{L^\infty(\mu_0)}$ and $\| \, \|_\theta$ with $\|v\|_\theta=\|v\|_{L^\infty(\mu_0)}+Lip(v)$, where $Lip(v)$ is
the Lipschitz constant of $v$ w.r.t.\ the distance $d_\theta(x,x'):=\theta^{s(x,x')}$.
\end{remark}

The main results in the present set-up are stated below.

\begin{thm}[finite measure]\label{prop-corel-finite} 
Assume (H0) a) and (H1).
Suppose that $v_X, w_X:X\to\R$ are such that $\|v_X\|_\theta^*<\infty$ and $\|w_X\|_{\infty}^*<\infty$.
Let $d\mu:=\frac{1}{\bar\varphi}d\mu_X$. Then
\begin{align*}
\int_X v_X w_X \circ f^n \, d\mu -\int_X v_X\, d\mu \int_X w_X  \, d\mu
&=\frac{1}{\bar\varphi}\sum_{j=n+1}^\infty \mu_0(\varphi>j)  \int_X v_X\, d\mu \int_X w_X  \, d\mu\\
&\quad +O(\|v_X \|_{\theta}^*\cdot \| w_X \|_{\infty}^*\cdot d_n),
\end{align*}
where
$$
d_n :=  \begin{cases}
n^{-\beta} &\text{ if } \beta > 2; \\
n^{-2} \log n &\text{ if } \beta = 2; \\
 n^{-(2\beta-2)} &\text{ if } \beta < 2.
\end{cases}
$$
\end{thm}

\begin{thm}[infinite measure]\label{prop-corel-infinite} 
Assume (H0) b) and (H1). 
Suppose that $v_X, w_X:X\to\R$ are such that  $\|v_X\|_\theta^*<\infty$ and  $\|w_X\|_{\infty}^*<\infty$. Let $q=\max\{j\geq 0:2(j+1)\beta>2j+1\}$. Then there exist
real constants\footnote{The constants $d_0,\ldots, d_{q-1}$ depend only on $\beta$ and the constant $c$
appearing in (H0) b). 
For their precise
form we refer to~\cite[Theorem 9.1]{MT} and~\cite[Theorem 1.1]{Terhesiu12}.}
$d_0,\ldots, d_{q-1}$ such that
\begin{align*}
\int_X v_X w_X \circ f^n \, d\mu_X 
&=(d_0n^{\beta-1}+\ldots + d_{q-1} n^{q(\beta-1)})\int_X v_X\, d\mu_X \int_X w_X  \, d\mu_X\\
&\quad +O(\|v_X\|_\theta^*\cdot \| w_X \|_{\infty}^*\cdot d_n),
\end{align*}
where $d_n= n^{-(\beta-1/2)}$.
\end{thm}
\begin{remark}
Instead of (H0) b), one can assume a stronger tail expansion of the form used in~\cite[Theorem 1.1]{Terhesiu12}
and as such obtain an improved error term in Theorem~\ref{prop-corel-infinite} . This is just an exercise, which can be solved 
by the argument used in the proof of Theorem~\ref{prop-corel-infinite}.
\end{remark}

\begin{remark}
The novelty of Theorem~\ref{prop-corel-finite} lies in the fact that
the observables are supported on the whole space, and of course the fact that this result can be obtained by inducing with general return times.
Theorem~\ref{prop-corel-infinite} is new, even for observables supported on $Y$.
\end{remark}

\section{Transfer operators on the tower}
\label{sec-gentower}

Let $L_\Delta$ be the transfer operator associated with
the tower map $T_\Delta$ and potential 
$$
p_\Delta(y,i) := \begin{cases}
0 & \text{ if } i < \varphi(y)-1,\\
p(y)& \text{ if } i = \varphi(y)-1.
\end{cases}
$$
Given that $L$ is the transfer operator associated with $(X,f,\mu_X)$, we have $L_\Delta \pi^* v_X = \pi^* L v_X$.

Recall that $Y   = Y_0$ is the base of the tower $\Delta$ and that $F=f^\varphi$ preserves the measure $\mu_0$. 
Choose $\theta \in (0,1)$ such that ~\eqref{eq:locLip} holds
and put $d_\theta(x,x'):=\theta^{s(x,x')}$, where $s(x, x')$
is the separation time.
Let $\B_\theta(Y)$ be the Banach space of $d_\theta$-Lipschitz functions $v: Y\to \R$
with norm $\|v\|_\theta=\|v\|_{L^\infty(\mu_0)}+Lip(v)$, where $Lip(v)$ is the Lipschitz constant of $v$ w.r.t.\ $d_\theta$.
 
Let $R^*:L^1(\mu)\to L^1(\mu)$ be  the transfer operator  associated with $F=f^\varphi$.
Under the assumption that $F$ is Gibbs Markov, it is known  that
 (see, for instance,~\cite[Section 5]{Sarig02}):
\begin{itemize}
\item[\textbf{(P1)}]
\begin{itemize}
\item[\textbf{a)}]
The space $\B_\theta(Y)$ contains constant functions and  $\mathcal{B}_\theta(Y)\subset L^\infty( \mu_0)$.
\item[\textbf{b)}]
 $1$ is a simple eigenvalue for $R^*$, isolated in
the spectrum of $R^*$.
\end{itemize}
\end{itemize}

Let $\D=\{z\in\C:|z|<1\}$ and
$\bar\D=\{z\in\C:|z|\le1\}$.
Given $z\in\bar\D$, we define $R^*(z):L^1(Y)\to L^1(Y)$ to be
the operator $R^*(z)v:= R^*(z^{\varphi}v)$.
 By (P1) b), $1$ is an isolated eigenvalue in the spectrum of $R^*(1)$. 
In addition, we know that (see, for instance,~\cite[Section 5]{Sarig02})

\begin{itemize}
\item[\textbf{(P2)}]
For $z\in\bar\D\setminus\{1\}$, the spectrum of $R^*(z)$ does not contain $1$.
\end{itemize}

Note that $\varphi$ is the first return time of $T_\Delta$ to the base $Y$. Define the following transfer operators that describe the general resp.\ the first return
to the base $Y$:
$$
T^*_nv := 1_{Y} L_\Delta^n(1_{Y}v),\enspace n\ge0, \qquad 
R_n^*v := 1_{Y}L_\Delta^n(1_{\{\varphi=n\}}v)=R^*(1_{\{\varphi= n\}} v),\enspace n\ge1.
$$

By, for instance,~\cite[Lemma 8]{Sarig02}
there is $C> 0$ such that 
\begin{itemize}
\item[\textbf{(P3)}] $\| R^*(1_a v ) \|_\theta \leq C \mu_0(a)\|1_a v\|_\theta$ for all $a \in \alpha$   
and hence $\|R^*_n\|_\theta=O(\mu_0(\varphi=n))$.
\end{itemize}
As recalled below, $(I-R^*(z))^{-1}$ can be used to understand the 
asymptotics
of the transfer operator of the Markov tower over the (general) induced map $F=f^\varphi$.
First, it is easy to see that 
$R^*(z)v= R^*(z^{\varphi} v)=\sum_{n\geq 1} (R_n^* v)z^n$. 
By (P1) and (P2), when viewed as a family of operators
acting on $\mathcal{B}_\theta(Y)\subset L^\infty(\mu_0)$, the function $(I-R^*(z))^{-1}$ 
is bounded 
and continuous on $\bar\D\setminus\{1\}$. By (P3), $R_n^*$ is bounded, so  $z\mapsto (I-R^*(z))^{-1}$ is analytic on $\D$.

Since $\varphi$ is a first return time of $T_\Delta$ to the base $Y$,
we have the renewal equation on the tower $T^*_n=\sum_{j=1}^{n} R_j^* \,T^*_{n-j}$.
For $z\in\bar\D$, define
 $T^*(z):=\sum_{n\geq 0} T^*_n z^n$ and recall  $R^*(z)=\sum_{n\geq 1} R_n^* z^n$. 
Since $(I-R^*(z))^{-1}$ is well defined on $\bar\D\setminus\{1\}$,
we have the following equation on $\bar\D\setminus\{1\}$:
\begin{equation}\label{eq-ren-tower}
T^*(z)=(I-R^*(z))^{-1}.
\end{equation}

Under both forms of (H0) (\ie finite and infinite measure preserving)
and the rest of the assumptions in Section~\ref{sec-frame},
the asymptotic behavior of $T_n^*$ is well understood 
(\cite{Sarig02, Gouezel04, MT}); we also refer to Subsection~\ref{subsec-prev} where we recall these results.

Since $v_\Delta = v_X \circ \pi$ is in general not supported on $Y_0$,
equation \eqref{eq-ren-tower} cannot be used as such to obtain information on the asymptotic behaviour of the correlation function
on the tower given by
$\int_\Delta v_\Delta w_\Delta \circ T_\Delta^n \, d\mu_\Delta=\int_\Delta L_\Delta^n v_\Delta w_\Delta\, d\mu_\Delta$. 
Hence, one cannot just use  \eqref{eq-ren-tower} and ~\eqref{eq:muDeltamuX} to estimate the correlation function
$\int_X v_X  w_\Delta \circ f^n d\mu_X$ for the map $f:X \to X$. However, 
the operators $A_n$ and $B_n$ defined below
can be used to deal with precisely this problem.

Following \cite[Section 2.1.1]{GouezelPhD}, to understand the behaviour of $L_\Delta^n$ via 
the behaviour of $T^*_n$, we need to define several operators that describe the action of 
$L_\Delta^n$ outside $Y$.
Recall from~\eqref{eq-genDelta} that $\varphi_\Delta(y,i)=\varphi(y)-i$ and define the operators
associated with the end resp.\ beginning
of an orbit on the tower 
as
$$
A_nv :=  L_\Delta^n(1_{\{ \varphi > n \} } v ), \quad  n \geq 0,
\qquad \qquad
B_nv := \begin{cases}
1_{Y} L_\Delta^n(1_{\{ \varphi_\Delta = n  \} \setminus Y }v ), &  n \geq 1, \\
1_{Y}v, & n = 0.
\end{cases}
$$
The operator associated with orbits that do not see the base of the tower is:
$$
C_n v :=  L_\Delta^n(1_{\{ \varphi_\Delta > n \}\setminus Y } v), \quad n \geq 0.
$$
As noticed in  \cite[Section 2.1.1]{GouezelPhD}, the following equation describes
the relationship between $T^*_n=1_{Y} L_\Delta^n 1_{Y}$ and $L_\Delta^n$.
\begin{equation}\label{eq-LDelta}
  L_\Delta^n = \sum_{n_1+n_2+n_3=n} A_{n_1}T^*_{n_2}B_{n_3} + C_n.
\end{equation}

\section{Proofs of the main results: previous and new ingredients}
\label{sec-strategy}

\subsection{Previous ingredients}
\label{subsec-prev}

As already mentioned,  $T_\Delta^\varphi$ is the first return map for $T_\Delta$
to the base $Y_0=Y$. Thus,
previous results on  renewal theory, in particular ~\cite[Theorem 1]{Gouezel04} (under (H0) a)-finite measure case) and
~\cite[Theorem 9.1]{MT} (under (H0) b)- infinite measure case), apply to $T_n^*$.
We start by recalling these results, as relevant to the present setting.
Let $P$ denote the spectral projection corresponding to the eigenvalue $1$
for $R^*(1)$. So, we can write $Pv(y)\equiv \int_Y v\, d\mu_0$.

\begin{lemma}~\cite[Theorem 1]{Gouezel04}
\label{lemma-Gouezel}
Assume that $F$ is Gibbs Markov and that (H0) a) holds. Then
\[
T_n^*=\frac{P}{\bar\varphi}+\frac{1}{\bar\varphi^2}\sum_{k=n+1}^\infty \sum_{j=k+1}^\infty P R_j P +E_n,
\]
where $E_n$ is an operator on $\B_\theta(Y)$ satisfying 
$$
\|E_n\|_\theta\ll \begin{cases}
n^{-\beta} &\text{ if } \beta > 2; \\
n^{-2} \log n &\text{ if } \beta = 2; \\
 n^{-(2\beta-2)} &\text{ if } \beta < 2.
\end{cases}
$$
\end{lemma}

\begin{lemma}~\cite[Theorem 9.1]{MT}\label{lemma-MT}
Assume that  $F$ is Gibbs Markov  and that (H0) b)holds. Let $q=\max\{j\geq 0:2(j+1)\beta>2j+1\}$. Then there exist
real constants $d_0,\ldots, d_{q-1}$ (depending only on the constants and parameters involved in (H0) b))\footnote{For the precise
form of these constants we refer to~\cite[Theorem 9.1]{MT} and~\cite[Theorem 1.1]{Terhesiu12}.}
such that
\[
T_n^*=(d_0n^{\beta-1}+\ldots d_{q-1} n^{q(\beta-1)})P +D_n, \qquad n \geq 1,
\] 
where $D_n$  is an operator on $\B_\theta(Y)$ satisfying $\|D_n\|_\theta=O(n^{-(\beta-1/2)})$.
\end{lemma}

\subsection{New ingredients: estimates related to $A_n, B_n, C_n$ under (H1)}
\label{subsec-new}

In the results stated below  we use the norm $\|\,\|_{\theta}^*$ and $\|\,\|_{\infty}^*$  as 
defined in~\eqref{eq-extracond2}. Recall that $v_\Delta = v_X \circ \pi$ and 
 $w_\Delta = w_X \circ \pi$. The proofs of the following results are postponed to Subsection~\ref{subsec-someproofs}.

\begin{lemma}~\label{lemma-Cn}
Assume that (H1) holds.
Let $v_X, w_X:X\to\R$ such that $\|v_X\|^*_{\infty}<\infty$ and  $ \| w_X\|_{L^\infty(\mu_X)}<\infty$. Then there exists $C>0$ such that for any $n\geq 0$,
\[
| \int_\Delta C_n v_\Delta w_\Delta \, d\mu_\Delta |\leq C \mu_0(\varphi>n) \|v_X\|^*_{\infty}   \| w_X\|_{L^\infty(\mu_X)}.
\]
\end{lemma}

\begin{lemma}\label{lemma-An}
Assume that (H1) holds. Let $v_Y:Y \to \R$
and $w_X:X\to\R$ be such that $\| v_Y \|_{L^\infty(\mu_0)}<\infty$  and $\|w_X\|^*_{\infty}<\infty$. Then there exists $C>0$ such that for any $n\geq 0$,
\[
\Big|\int_\Delta \sum_{j\geq n} A_j v_Y  w_\Delta \, d\mu_\Delta\Big|\leq C \mu_0(\varphi>n) \|v_Y\|_{L^\infty(\mu_0)}  \| w_X \|_{\infty}^*.
\]
\end{lemma}

For the statement below we note that  by definition, $B_nv_\Delta$ is a function supported on the base $Y$ of the tower $\Delta$,
so $\|B_nv_\Delta\|_\theta$ makes sense.

\begin{lemma}\label{lemma-Bn2} Assume that (H1) holds. Let $v_X:X\to\R$ such that $\|v_X\|_{\theta}^*<\infty$.
Then there exists $C>0$ such that for any $n\geq 0$,
\[
\sum_{j \geq n} \| B_j v_\Delta\|_\theta \leq C \mu_0(\varphi > n) \| v_X\|_\theta^* .
\]

\end{lemma}

\begin{remark}\label{rmk-suppY2}
Continuing from Remark~\ref{rmk-suppY1}, we note that
under the assumption that $v_X, w_X$ are supported on $Y$, all the statements 
in this subsection simplify since the weighted norms 
$\|\,\|_{\theta}^*$ and $\|\,\|_{\infty}^*$ coincide with $\|\,\|_{\theta}$ and $\|\,\|_{L^\infty(\mu_0)}$.
Moreover, under this assumption, the proofs in this subsection
are simplified, although the assumption (H1) is still required. In Subsection~\ref{subsec-someproofs} we point out such a simplification for 
the proof of Lemma~\ref{lemma-Cn}: see Remark~\ref{rmk-suppY3}.
\end{remark}

\subsection{Proofs  of  Theorem~\ref{prop-corel-finite} and Theorem~\ref{prop-corel-infinite}}
\label{subsec-strat}

From the statement of Theorem~\ref{prop-corel-finite}, recall that $d\mu=\frac{1}{\bar\varphi} d\mu_X$ is the 
normalized $f$-invariant measure (when $\bar\varphi<\infty$).
By~\eqref{eq:muDeltamuX} and the definition of $\mu$, in order to estimate the correlation function for observables on the space $X$,
in the \emph{finite} measure case it suffices to estimate
\begin{equation}\label{eq-opcorelfin}
\int_X v_X w_X \circ f^n \, d\mu=
\frac{1}{\bar\varphi}\int_\Delta v_\Delta w_\Delta \circ T_\Delta^n \, d\mu_\Delta=
\frac{1}{\bar\varphi}\int_\Delta L_\Delta^n v_\Delta w_\Delta  \, d\mu_\Delta.
\end{equation}
Similarly, due to~\eqref{eq:muDeltamuX} in the \emph{infinite} measure case we estimate
\begin{equation}
\label{eq-opcorelinf}
\int_X v_X w_X \circ f^n \, d\mu_X=\int_\Delta L_\Delta^n v_\Delta w_\Delta  \, d\mu_\Delta.
\end{equation}
By equation~\eqref{eq-LDelta} and Lemma~\ref{lemma-Cn},
\begin{align}
\label{eq-LnCn}
\int_\Delta L_\Delta^n v_\Delta  w_\Delta \, d\mu_\Delta=
\int_\Delta\sum_{n_1+n_2+n_3=n} A_{n_1}T^*_{n_2}B_{n_3}v_\Delta  w_\Delta \, d\mu_\Delta + O(\mu_0(\varphi>n) \|v_X\|^*_{\infty}  \| w_X \|_{L^\infty(\mu_X)}).
\end{align}
In order to take advantage of the full force of Lemmas~\ref{lemma-Gouezel} and~\ref{lemma-MT} along with the estimates related to $A_n, B_n$ in Subsection~\ref{subsec-new}
we define
\begin{align*}
 A(z):=\sum_{n\geq 0} A_n z^n, \quad B(z):=\sum_{n\geq 0} B_n z^n,\qquad z\in\D
\end{align*}
and recall that when viewed as a family of operators acting on $\B_\theta(Y)\subset L^\infty(\mu_0)$,
$T^*(z)=\sum_{n\geq 0} T^*_n z^n$ is well defined on $\D$ (in fact on $\bar\D$, but we do not use this
information below). 
By  definition, the operator sequences
$A_n: L^\infty(\mu_0)\to L^1(\mu_\Delta)$,
$B_n: L^\infty(\mu_\Delta)\to L^\infty(\mu_0)$, $T_n^*: L^\infty(\mu_0)\to L^\infty(\mu_0)$  are bounded. 
 As a consequence,
$ A(z)T^*(z)B(z), z\in\D$,
is well defined as a family of operators from $L^\infty(\mu_\Delta)$ to $L^1(\mu_\Delta)$. Given this we can view 
\[
\int_\Delta G_n v_\Delta  w_\Delta \, d\mu_\Delta:=\int_\Delta\sum_{n_1+n_2+n_3=n} A_{n_1}T^*_{n_2}B_{n_3}v_\Delta  w_\Delta \, d\mu_\Delta
\]
as the $n$-th coefficient of $\int_\Delta G(z) v_\Delta  w_\Delta \, d\mu_\Delta:=\int_\Delta A(z)T^*(z)B(z) v_\Delta  w_\Delta \, d\mu_\Delta, z\in\D$.

In the sequel we also use the following statement on $A(1)$ and $B(1)$ that does not require (H1); the statement on $B(1)$ below relies on the fact that
$B(1) v_\Delta$ is a function supported on $Y$.
\begin{lemma}\label{lemma-of1} We have
$$
\begin{cases}
\int_\Delta A(1) 1_{Y} w_\Delta d\mu_\Delta 
= \int_\Delta w_\Delta d\mu_\Delta=\int_X w_X d\mu_X,
\\[2mm]
\int_\Delta B(1) v_\Delta d\mu_\Delta =\int_Y B(1) v_\Delta d\mu_0=
 \int_\Delta v_\Delta d\mu_\Delta=\int_X v_X d\mu_X.
\end{cases}
$$
\end{lemma}

\subsubsection{Finite measure case}

 Let $E_n:\B_\theta(Y)\to \B_\theta(Y)$ be as in the statement of
Lemma~\ref{lemma-Gouezel} and put $E(z)=\sum_n E_n z^n$, $z\in\D$. 
By  Lemma~\ref{lemma-Gouezel},
\begin{equation}\label{eq-decomp}
\int_\Delta  G(z) v_\Delta w_\Delta  \, d\mu_\Delta =I_{main}(z)(v_\Delta, w_\Delta) +I_E(z)(v_\Delta, w_\Delta)
\end{equation}
for
$$
\begin{cases}
I_{main}(z) (v_\Delta, w_\Delta):=\\\quad\quad  \frac{1}{1-z} \int_\Delta A(z) \Big(\frac{P}{\bar\varphi}+(1-z)\sum_{n=0}^\infty( \frac{1}{\bar\varphi^2}
\sum_{k=n+1}^\infty \sum_{j=k+1}^\infty P R_j P) z^n\Big)
B(z) v_\Delta w_\Delta  \, d\mu_\Delta,\\[2mm]
 I_E(z)(v_\Delta, w_\Delta) := \int_\Delta A(z) E(z) B(z) v_\Delta w_\Delta  \, d\mu_\Delta.
\end{cases}
$$
By the above, in order to estimate $\int_\Delta G_n v_\Delta  w_\Delta \, d\mu_\Delta$, we need to estimate the coefficients of the functions
$I_{main}(z)(v_\Delta, w_\Delta)$ and $I_E(z)(v_\Delta, w_\Delta)$, $z\in\D$ for
appropriate $v_\Delta, w_\Delta$ (equivalently for appropriate $v_X, w_X$).
Let 
\begin{align}\label{eq-Vz}
 V(z)(v_\Delta, w_\Delta) :=
\int_\Delta A(1)\Big(\sum_{n=0}^\infty( \frac{1}{\bar\varphi^2}\sum_{k=n+1}^\infty \sum_{j=k+1}^\infty P R_j P) z^n\Big) B(1) v_\Delta w_\Delta  \, d\mu_\Delta
\end{align}
and note  that
\begin{align}\label{eq-decomp_main}
\nonumber I_{main}(z)(v_\Delta, w_\Delta)-&\frac{1}{1-z} \int_\Delta A(1) \frac{P}{\bar\varphi}B(1) v_\Delta w_\Delta  \, d\mu_\Delta-V(z)(v_\Delta, w_\Delta) \\
&= I^A(z)(v_\Delta, w_\Delta) + I^B(z)(v_\Delta, w_\Delta)
\end{align}
for
$$
\begin{cases}
I^A(z)(v_\Delta, w_\Delta) :=\int_\Delta \frac{A(z)-A(1)}{1-z}  \Big(\frac{P}{\bar\varphi}+
(1-z)\sum_{n=0}^\infty( \frac{1}{\bar\varphi^2}\sum_{k=n+1}^\infty \sum_{j=k+1}^\infty 
P R_j P) z^n\Big)B(1) v_\Delta w_\Delta  \, d\mu_\Delta\\[2mm]
I^B(z)(v_\Delta, w_\Delta) := \int_\Delta A(1)\Big(\frac{P}{\bar\varphi}+
(1-z)\sum_{n=0}^\infty( \frac{1}{\bar\varphi^2}\sum_{k=n+1}^\infty \sum_{j=k+1}^\infty P R_j P) z^n\Big)\frac{B(z)-B(1)}{1-z}  v_\Delta w_\Delta  \, d\mu_\Delta\end{cases}
$$
Below we provide the estimates obtained in the sequel for the coefficients of the terms
in~\eqref{eq-Vz} and~\eqref{eq-decomp_main} and as such complete

\begin{proof}[Proof of Theorem~\ref{prop-corel-finite}]
By Lemma~\ref{lemma-of1}, 
\[
\frac{1}{1-z}\int_\Delta A(1)\frac{P}{\bar\varphi} B(1)
 v_\Delta w_\Delta  \, d\mu_\Delta=\frac{1}{\bar\varphi}\int_X v_X\,d\mu_X\int_X w_X\,d\mu_X\sum_{n=0}^\infty z^n.
\] 
By Lemma~\ref{lemma-first2main}, the coefficients $V_n(v_\Delta, w_\Delta)$ of $V(z)(v_\Delta, w_\Delta)$, $z\in\D$ are given by
\[
V_n(v_\Delta, w_\Delta)=\frac{1}{\bar\varphi^2}\sum_{k=n+1}^\infty \mu_0(\varphi>k)  \int_X v_X\, d\mu_X \int_X w_X  \, d\mu_X.
\]
We continue with the estimates for  the coefficients of the terms
in the RHS of~\eqref{eq-decomp_main}.
 Lemmas~\ref{lemma-thirdmain}
and~\ref{lemma-lastmain} together with (H0) a) imply that the coefficients of the functions $I^A(z)(v_\Delta, w_\Delta)$ and $I^B(z)(v_\Delta, w_\Delta)$, $z\in\D$
are  $O(\| v_X \|_{\theta}^*\, \| w_X \|_{\infty}^*\,\mu_0(\varphi>n))=O(n^{-\beta}\| v_X \|_{\theta}^*\, \| w_X \|_{\infty}^*)$.

It remains to estimate the  coefficients
$\int_\Delta \sum_{n_1+n_2+n_3=n} A_{n_1}E_{n_2}B_{n_3} v_\Delta w_\Delta  \, d\mu_\Delta$ of the function $I_E(z)$, $z\in\D$.
By Lemma~\ref{lemma-Bn2} and (H0) a), $\|B_nv_\Delta\|_\theta=O(\mu_0(\varphi > n)\|v_X\|_\theta^*)=O(n^{-\beta}\|v_X\|_\theta^*)$.
Hence,  the convolution of $E_n$ and $B_n$ satisfies
$\|\sum_{n_2+n_3=n}E_{n_2}B_{n_3} v_\Delta\|_\theta=O(\|E_n\|_\theta \cdot \|v_X\|_\theta^*)$. 
Next,  given that $v_Y$ is a function supported on the base $Y$, the definition of $A_n$ and (H0) a) gives
\[
|\int_\Delta A_n v_Y w_\Delta \, d\mu_\Delta| \ll \mu_0(\varphi > n) \|w_\Delta\|_{L^\infty(\mu_\Delta)}\|v_Y\|_{L^\infty(\mu_0)}
\ll n^{-\beta} \|w_X\|_{L^\infty(\mu_X)}\|v_Y\|_\theta.
\]
Hence, $|\int_\Delta \sum_{n_1+n_2+n_3=n} A_{n_1}E_{n_2}B_{n_3} v_\Delta w_\Delta  \, d\mu_\Delta|=O(\|E_n\|_\theta \cdot \|v_X\|_\theta^*\cdot \|w_X\|_{L^\infty(\mu_X)})$.

Putting the above together and using~\eqref{eq-decomp_main} and~\eqref{eq-LnCn},
\begin{align*}
\int_\Delta v_\Delta w_\Delta \circ T_\Delta^n \, d\mu_\Delta&=\frac{1}{\bar\varphi}\int_X v_X\,d\mu_X\int_X w_X\,d\mu_X
+\frac{1}{\bar\varphi^2}\sum_{k=n+1}^\infty \mu_0(\varphi>k)  \int_X v_X\, d\mu_X \int_X w_X  \, d\mu_X\\
&\quad +O(\|E_n\|_\theta \cdot \|v_X\|_\theta^*\cdot \|w_X\|_\infty^*).
\end{align*}
The conclusion follows recalling that $d\mu=\frac{1}{\bar\varphi}d\mu_X$ and using equation~\eqref{eq-opcorelfin}.~\end{proof}

{\bf Infinite measure case.} 
Write
\begin{equation}\label{eq-decomp1}
\int_\Delta G(z)  v_\Delta w_\Delta \, d\mu_\Delta -
\int_\Delta A(1)T^*(z)B(1)
 v_\Delta w_\Delta  \, d\mu_\Delta = I_{\inf}^A(z) + I_{\inf}^B(z)
\end{equation}
for
$$
\begin{cases}
 I_{\inf}^A(z)(v_\Delta, w_\Delta) := \int_\Delta (A(z)-A(1)) T^*(z) B(1) v_\Delta w_\Delta  \, d\mu_\Delta,\\[2mm]
I_{\inf}^B(z)(v_\Delta, w_\Delta) := \int_\Delta A(1)T^*(z) (B(z)-B(1)) v_\Delta w_\Delta  \, d\mu_\Delta.
\end{cases}
$$
Below, we provide the estimates obtained in the sequel for the coefficients of the terms in~\eqref{eq-decomp1} and as such
complete
\begin{proof}[Proof of Theorem~\ref{prop-corel-infinite}]
By Lemma~\ref{lemma-MT}, the $n$-th coefficient $[\int_\Delta A(1)T^*B(1)
 v_\Delta w_\Delta  \, d\mu_\Delta]_n$ of the function $\int_\Delta A(1)T^*(z)B(1)
 v_\Delta w_\Delta  \, d\mu_\Delta$, $z\in\D$,  satisfies
\begin{align*}
[\int_\Delta A(1)T^*B(1)
 v_\Delta w_\Delta  \, d\mu_\Delta]_n&=
(d_0n^{\beta-1}+\ldots +d_{q-1} n^{q(\beta-1)})\int_\Delta A(1)P B(1)
 v_\Delta w_\Delta  \, d\mu_\Delta\\
&+\int_\Delta A(1)D_n B(1) v_\Delta w_\Delta  \, d\mu_\Delta,
\end{align*}
where $\|D_n\|_\theta=O(n^{-(\beta-1/2)})$. 
By Lemma~\ref{lemma-Bn2}, $\|B(1)v_\Delta\|_\theta\leq C\| v_X \|_\theta^*$  for some $C>0$. Hence,
\begin{align*}
|\int_\Delta A(1)D_n B(1) v_\Delta w_\Delta  \, d\mu_\Delta|&=O(n^{-(\beta-1/2)}\|w_\Delta\|_{L^\infty(\mu_\Delta)}
\|B(1)v_\Delta\|_\theta)\\
&=O(n^{-(\beta-1/2)}\|w_X\|_{L^\infty(\mu_X)}
 \| v_X \|_\theta^*).
\end{align*}
By Lemma~\ref{lemma-of1}, $\int_\Delta A(1)P B(1)
 v_\Delta w_\Delta  \, d\mu_\Delta=\int_X v_X\,d\mu_X \int_X w_X\,d\mu_X$. Putting these together,
\begin{align*}
[\int_\Delta A(1)T^*(z)B(1)
 v_\Delta w_\Delta  \, d\mu_\Delta]_n&=
(d_0n^{\beta-1}+\ldots +d_{q-1} n^{q(\beta-1)})\int_X v_X\,d\mu_X \int_X w_X\,d\mu_X\\
&+O(n^{-(\beta-1/2)}\|w_X\|_{L^\infty(\mu_X)}
 \| v_X \|_\theta^*).
\end{align*}

By Lemma~\ref{lemma-IAn-IBn} the coefficients of the functions $I_{\inf}^A(z)(v_\Delta, w_\Delta)$ and 
$I_{\inf}^B(z)(v_\Delta, w_\Delta)$, $z\in\D$, are  $O(n^{-\beta} \|v_X\|_{\theta}^*\|w_X\|_{\infty}^*)$.

Putting the above together and using equations~\eqref{eq-decomp1} and~\eqref{eq-LnCn},
\begin{align*}
\int_\Delta v_\Delta w_\Delta \circ T_\Delta^n \, d\mu_\Delta&=(d_0n^{\beta-1}+\ldots +d_{q-1} n^{q(\beta-1)})\int_X v_X\,d\mu_\Delta\int_X w_X\,d\mu_X \\
& \quad +O(n^{-\beta}\| v_X \|_\theta^*\,\|w_X\|_{\infty}^*).
\end{align*}
The conclusion follows from the above equation together with equation~\eqref{eq-opcorelinf}. ~\end{proof}

\subsection{Proofs of the lemmas in Subsections~\ref{subsec-new} and~\ref{subsec-strat}}
\label{subsec-someproofs}

\begin{proof}[Proof of Lemma~\ref{lemma-Cn}]
Recall that $\tau_0(y) = 0$ and 
$\tau_k(y) = \tau_{k-1}(y)+\tau(f^{\tau_{k-1}(y)}(y))$ is the $k$-th return time
to $Y$. Compute for $y \in Y$
\begin{align}\label{eq:w*}
\sum_{j=\tau_k(y)}^{\tau_{k+1}(y)-1} & | v_\Delta \circ T^j_\Delta(y,0)|
=
\sum_{j=0}^{\tau_{k+1}(y)-\tau_k(y)-1} | v_\Delta \circ T^j_\Delta(T^{\tau_k(y)}(y,0))|
\nonumber \\
&=
\sum_{j=0}^{\tau_{k+1}(y)-\tau_k(y)-1} | v_X \circ f^j(f^{\tau_k(y)}(y))|
\nonumber \\
&\leq \sum_{j=0}^{\tau_{k+1}(y)-\tau_k(y)-1} \| v_X \|^*_\infty (\tau_{k+1}(y)-\tau_k(y)-j)^{-(1+\eps)} \leq C_\eps  \|  v_X \|^*_\infty,
\end{align}
where $C_\eps = \sum_{j \geq 1} j^{-(1+\eps)}$. 
Thus,
\begin{align*}
\Big|\int_\Delta & C_n v_\Delta  w_\Delta \, d\mu_\Delta\Big|
= \Big|\int_\Delta L^n_\Delta(1_{\{ \varphi_\Delta > n\}\setminus Y}v_\Delta)  w_\Delta \, d\mu_\Delta \Big| \\
&\leq \int_{\{\varphi_\Delta > n\} \setminus Y}  |v_\Delta|   |w_\Delta \circ T^n_\Delta| \, d\mu_\Delta
\leq \| w_\Delta \|_{L^\infty(\mu_\Delta)} \int_{\{\varphi > n\} }  \sum_{j=1}^{\varphi(y)-n-1} |v_\Delta \circ T^j_\Delta| \, d\mu_0 \\
&\le
\| w_X \|_{L^\infty(\mu_X)}  \int_{\{ \varphi > n\} }  \sum_{k=0}^{\rho(y)-1} 
\sum_{j=\tau_k(y)}^{\tau_{k+1}(y)-1} |v_\Delta \circ T^j_\Delta| \, d\mu_0 \\
&\leq \| w_X\|_{L^\infty(\mu_X)}  C_\eps \| v_X\|^*_\infty \int_{\{\varphi > n\}} \rho(y)  \, d\mu_0 
\leq C C_\eps \mu_0(\varphi > n)  \| v_X\|^*_\infty\,  \| w_X\|_{L^\infty(\mu_X)},
\end{align*}
where the last inequality is obtained using (H1).
~\end{proof}

\begin{remark}\label{rmk-suppY3} Continuing from Remark~\ref{rmk-suppY2} we note the following.
If $v_X$ is supported on $Y$, then the sum $\sum_{j=\tau_k(y)}^{\tau_{k+1}(y)-1}  | v_\Delta \circ T^j_\Delta(y,0)|$
reduces to single term, namely $ | v_\Delta \circ T^{\tau_k(y)}_\Delta(y,0)|$. In this case the constant $C_\eps$ appearing in~\eqref{eq:w*}
disappears, but condition (H1) is still required.
\end{remark}

\begin{proof}[Proof of Lemma~\ref{lemma-An}]
Using \eqref{eq:w*} for $w_\Delta$ instead of $v_\Delta$ we find
\begin{align*}
\Big|\int_\Delta & \sum_{j=n}^\infty A_j v_Y  w_\Delta \, d\mu_\Delta\Big|
= \Big|\int_Y \sum_{j=n}^\infty 1_{\{ \varphi > j \} } v_Y  w_\Delta \circ T^j_\Delta \, d\mu_0\Big| \\
&\leq \int_{\{\varphi > n\}}  |v_Y| \ \sum_{j=n}^{\varphi(y)-1} 1_{\{\varphi > j\}}  |w_\Delta \circ T^j_\Delta| \, d\mu_0 \le
\|v_Y\|_{L^\infty(\mu_0)} \int_{\{\varphi > n\}}  \sum_{k=0}^{\rho(y)-1} 
\sum_{j=\tau_k(y)}^{\tau_{k+1}(y)-1} |w_\Delta \circ T^j_\Delta| \, d\mu_0 \\
&\le
\|v_Y\|_{L^\infty(\mu_0)} C_\eps \| w_X\|^*_\infty \int_{\{\varphi > n\}} \rho(y)  \, d\mu_0 
\leq C C_\eps \mu_0(\varphi > n) \|v_Y\|_{L^\infty(\mu_0)} \| w_X\|^*_\infty
\end{align*}
by (H1). This completes the proof.
\end{proof}

\begin{proof}[Proof of Lemma~\ref{lemma-Bn2}]
For the purpose of the argument below, we define weighted norms on the tower analogous to
\eqref{eq-extracond2}. Let
$\tau^*_\Delta(x) = 1+\min\{ i \geq 0 : T_\Delta^i(x) \in \hat Y\}$
for $\hat Y = \pi^{-1}(Y)$, and define
\begin{equation*}
\begin{cases}
\|v_\Delta \|_{\Delta,\infty}^*=\sup_{x\in \Delta}|v_\Delta(x)|\tau^*_\Delta(x)^{1+\epsilon}, \\[2mm]
|v_\Delta |_{\Delta,\theta}^* = \sup_{a \in \alpha} \sup_{0 \leq i < \varphi(a)} \sup_{x,x' \in a}
 \frac{\tau^*_\Delta(x,i) }{\theta^{s(x,x')}} |v_\Delta(x,i) - v_\Delta(x',i)|,
\end{cases}
\end{equation*}
and $\| v_\Delta \|_{\Delta, \theta}^* = \| v_\Delta \|_{\Delta,\infty}^* + 
| v_\Delta |_{\Delta,\theta}^*$.
In this way, $\| v_X \|_\theta^* = \| v_\Delta \|_{\Delta,\theta}^*$
whenever $v_\Delta = v_X \circ\pi$.
 
For $a \in \alpha$, $0 \leq j < \varphi(a)$, define
$B_{j,a}v_\Delta := L_\Delta^j(1_{\{ (y,i) : y \in a, i+j = \varphi(y)\}} v_\Delta)$.
The definition of $p_\Delta$ implies that for points on level $i$ of the tower, 
the potential $p_\Delta$ satisfies
$\sum_{j=0}^{\varphi(y)-i-1} p_\Delta \circ T^j_\Delta(y,i) = p(y)$.
Writing $y = F^{-1}(x) \cap a$ and  $y' = F^{-1}(x') \cap a$, we compute
using \eqref{eq:locLip}
\begin{eqnarray*}
\| B_{j,a}v_\Delta\|_{\theta} &=&
\sup_{x \in Y} e^{p(y)} |v_\Delta(y,\varphi(y)-j)| \\
&& \quad + 
\sup_{x,x' \in Y} \theta^{-s(x,x')} \left|e^{p(y)} v_\Delta(y, \varphi(y)-j)
- e^{p(y')} v_\Delta(y', \varphi(y')-j) \right| \\
&\le&
\mu_0(a) (\tau^* \circ f^{\varphi(a)-j}(a) )^{-(1+\eps)} \|v_\Delta\|_{\Delta,\infty}^* \\
&& \quad + 
\sup_{x,x' \in Y} \theta^{-s(x,x')} \left( |e^{p(y)} - e^{p(y')}| 
 (\tau^* \circ f^{\varphi(a)-j}(a) )^{-(1+\eps)} \|v_\Delta\|_{\Delta,\infty}^* \right. \\
&& \quad + \left. \mu_0(a)  (\tau^* \circ f^{ \varphi(a)-j }(a) )^{-(1+\eps)} 
|v_\Delta |_\theta^* \right)
\\
&\le& C_1 \mu_0(a) (\tau^* \circ f^{ \varphi(a)-j }(a) )^{-(1+\eps)} \|v_\Delta\|_{\Delta,\theta}^*.
\end{eqnarray*}

Thus,
\begin{align*}
\sum_{j \geq n} & \| B_jv_\Delta \|_\theta 
\leq \sum_{\stackrel{a \in \alpha}{\varphi(a) > n}} \sum_{j=1}^{\varphi(a)-1} \| B_{j,a}v_\Delta \|_\theta  
\leq \sum_{\stackrel{a \in \alpha}{\varphi(a) > n}} 
\sum_{k=0}^{\rho(a)-1} \sum_{j=\tau_k(a)+1}^{\tau_{k+1}(a)} \| B_{j,a}v_\Delta \|_\theta  \\
&\leq \sum_{\stackrel{a \in \alpha}{\varphi(a) > n}} 
C_1 \mu_0(a) \sum_{k=0}^{\rho(a)-1} \sum_{j=\tau_k(a)+1}^{\tau_{k+1}(a)} 
(1+\tau_{k+1}(a)-j)^{-(1+\eps)}  \| v_\Delta \|^*_{\Delta,\theta} \\
&\leq  C_1 C_\eps\| v_\Delta \|^*_{\Delta,\theta}\sum_{\stackrel{a \in \alpha}{\varphi(a) > n}} 
 \mu_0(a) \rho(a) = C_1 C_\eps \| v_\Delta \|^*_{\Delta,\theta}
\int_{\{\varphi > n\}} \rho(y) \, d\mu_0 \\
&\leq C_1 C_\eps C\, \| v_\Delta \|^*_{\Delta,\theta} \, \mu_0(\varphi > n),
\end{align*}
where the last inequality was obtained using  (H1).
\end{proof}

\begin{proof}[Proof of Lemma~\ref{lemma-of1}]
By direct computation:
\begin{eqnarray*}
\int_\Delta A(1) 1_{Y} w_\Delta d\mu_\Delta &=&
\int_\Delta \sum_{n \geq 0}^\infty  
L_\Delta^n(1_{\{ \varphi > n\}}) w_\Delta d\mu_\Delta 
=\int_\Delta \sum_{n \geq 0}^\infty  
1_{\{ \varphi > n\}} w_\Delta \circ T_\Delta^n\, d\mu_\Delta \\
&=&\int_{Y} \sum_{n \geq 0}^{\varphi-1} w_\Delta \circ T_\Delta^n d\mu_0
= \int_{\Delta} w_\Delta d\mu_\Delta.
\end{eqnarray*}
The statement on $A$ follows. For the statement on $B$, let $v_k = 1_{\{\varphi_\Delta = k\}}v_\Delta$ for $k  \geq 0$ and set $v_{k,j} = 1_{Y_j}v_k$
where $Y_j$ is the $j$-th level of the tower.
Then $v_{k,j}$ have disjoint supports, and  for each element $a \in \alpha$, there is only one $j$ such that $v_{k,j}$
is supported on $f_\Delta^j(a)\subset Y_j$, namely, $k+j = \varphi(a)$.
Let $u_k(y,0) = v_k(y,j)$ and compute that
\begin{equation}\label{eq-Bn}
L_\Delta^kv_k = \sum_{j=1}^\infty L_\Delta^k v_{k,j} = \sum_{j = 1}^\infty L_\Delta^{k} L_\Delta^{j} u_k
= L_\Delta^\varphi u_k=R^*u_k,
\end{equation}
Recall that $B(1)v_\Delta$ is supported on $Y$. Hence,
\begin{align*}
\int_Y B(1) v_\Delta \, d\mu_0 &= \int_{Y} \sum_{k=0}^\infty L_\Delta^k v_k \, d\mu_0
= \int_{Y} \sum_{k=0}^\infty R^* u_k  \, d\mu_0
 = \int_{Y} \sum_{k=1}^\infty u_k  \, d\mu_0 =
\int_\Delta v_\Delta \, d\mu_\Delta.
\end{align*}
The conclusion follows since $\int_\Delta v_\Delta \, d\mu_\Delta=\int_X v_X \, d\mu_X$.
\end{proof}

\section{Proofs of  Lemmas used in the proof of Theorem~\ref{prop-corel-finite} (finite measure case)}
\label{sec-proofs_finite}

\subsection{Estimating the coefficients of $V(z)(v_\Delta, w_\Delta)$ defined in~\eqref{eq-Vz} }

\begin{lemma}\label{lemma-first2main}Assume the setting of  Lemma~\ref{lemma-Gouezel}.
Then the coefficients $V_n(v_\Delta, w_\Delta)$ of the function $V(z)(v_\Delta, w_\Delta)$, $z\in\D$ are given by
\[
V_n(v_\Delta, w_\Delta)=\frac{1}{\bar\varphi^2}\sum_{k=n+1}^\infty \mu_0(\varphi>k)  \int_X v_X\, d\mu_X \int_X w_X  \, d\mu_X.
\]
\end{lemma}

\begin{proof}
By Lemma~\ref{lemma-of1}, $\int_Y B(1) v_\Delta d\mu_0 =
 \int_X v_X\, d\mu_X $. Recalling $R_n^*v=1_{Y}L_\Delta^n(1_{\{\varphi=n\}}v)$,
we compute that
\[
\Big(\sum_{n=0}^\infty( \frac{1}{\bar\varphi^2}\sum_{k=n+1}^\infty \sum_{j=k+1}^\infty P R_j^* P) z^n\Big)B(1) v_\Delta
=\frac{1}{\bar\varphi^2}\int_X v_X\, d\mu_X\sum_{n=0}^\infty \sum_{k=n+1}^\infty\mu_0(\varphi>k) z^n.
\]
Thus, using Lemma~\ref{lemma-of1} (first, for the statement on $B$ and at the end of the argument for $A$),
\begin{align*}\label{eq-secmain}
\nonumber V(z)(v_\Delta, w_\Delta)&=\int_\Delta A(1)\Big(\sum_{n=0}^\infty(\frac{1}{\bar\varphi^2}\sum_{k=n+1}^\infty \sum_{j=k+1}^\infty P R_j^* P) z^n\Big) B(1) v_\Delta w_\Delta  \, d\mu_\Delta\\
\nonumber &=\frac{1}{\bar\varphi^2}\int_X v_X\, d\mu_X\sum_{n=0}^\infty \sum_{k=n+1}^\infty\mu_0(\varphi>k) z^n\int_\Delta A(1) 1_Y w_\Delta  \, d\mu_\Delta\\
&=\frac{1}{\bar\varphi^2}\int_X v_X\, d\mu_X\, \int_X w_X d\mu_X\, \sum_{n=0}^\infty \sum_{k=n+1}^\infty\mu_0(\varphi>k) z^n.
\end{align*}
The conclusion follows.~\end{proof}

\subsection{Estimating the coefficients of $I^{A}(z)(v_\Delta, w_\Delta)$ and $I^{B}(z)(v_\Delta, w_\Delta)$ defined in~\eqref{eq-decomp_main} }

We begin with some immediate consequences of Lemmas~\ref{lemma-An} and~\ref{lemma-Bn2}.
\begin{lemma}\label{lemma-Az-A1}
Assume (H0) a) and (H1). Let $w_X:X\to\R$ such that $\|w_X\|_{\infty}^*<\infty$. Then
\[
\int_\Delta \frac{A(z)-A(1)}{1-z} 1_Y  w_\Delta \, d\mu_\Delta=\sum_{n\geq 0} a_n z^n,
\]
where $|a_n|=O(\mu_0(\varphi>n)  \| w_X \|_{\infty}^*)$.
\end{lemma}

\begin{proof}
Compute that
\begin{align*}
-\int_\Delta & \frac{A(z)-A(1)}{1-z}  1_Y  w_\Delta \, d\mu_\Delta
=\int_\Delta \sum_{n\geq 0} \sum_{j\geq n} A_j 1_Y  w_\Delta z^n \, d\mu_\Delta.
\end{align*}
The conclusion follows from the above equation together with Lemma~\ref{lemma-An}.~\end{proof}
\begin{lemma}\label{lemma-Vz-V1}
Assume (H0) a) and (H1). Let $v_X:X\to\R$ such that $\|v_X\|_{\theta}^*<\infty$. Then
\[
\int_\Delta \frac{B(z)-B(1)}{1-z} v_\Delta \, d\mu_\Delta=\sum_{n\geq 0} a_n z^n,
\]
where $|a_n|=O(\mu(\varphi>n)  \|v_X\|_{\theta}^*)$.
\end{lemma}
\begin{proof}The conclusion follows by the argument used in the proof of Lemma~\ref{lemma-Az-A1}, using Lemma~\ref{lemma-Bn2}
instead of Lemma~\ref{lemma-An}.
\end{proof}

The coefficients of  $I^A(z)$ will be obtained by decomposing this  term into
$I^A(z)=D_A(z)+F_A(z)$, where
\[
D_A(z)(v_\Delta, w_\Delta):=\int_\Delta \frac{A(z)-A(1)}{1-z} \frac{P}{\bar\varphi}B(1) v_\Delta w_\Delta \, d\mu_\Delta
\]
and
\[
F_A(z)(v_\Delta, w_\Delta):=\int_\Delta \frac{A(z)-A(1)}{1-z}  Q(z)B(1)v_\Delta w_\Delta \, d\mu_\Delta,
\]
with
\begin{equation}\label{eq-Q}
Q(z) v:=(1-z)\sum_{n=0}^\infty( \frac{1}{\bar\varphi^2}\sum_{k=n+1}^\infty \sum_{j=k+1}^\infty P R_j^* P) z^n\Big)v.
\end{equation}

\begin{lemma}\label{lemma-thirdmain}Assume the setting of Lemma~\ref{lemma-Gouezel}. Assume that (H1) holds.
Suppose that $v_X, w_X:X\to\R$ are such that$\| v_X \|_{L^\infty(\mu_X)}<\infty$ and $\|w_X\|_{\infty}^*<\infty$. 
Then the coefficients $D_{A,n}(v_\Delta, w_\Delta)$, $F_{A,n}(v_\Delta, w_\Delta)$ of the functions $D_A(z)(v_\Delta, w_\Delta)$ and $F_A(z)(v_\Delta, w_\Delta)$, $z\in\D$
 satisfy
\[
\begin{cases}
|D_{A,n}(v_\Delta, w_\Delta)|= O(\mu_0(\varphi>n)   \| v_X \|_{L^\infty(\mu_X)} \| w_X \|_{\infty}^*),\\[1mm] 
\, |F_{A,n}(v_\Delta, w_\Delta)|= O(\mu_0(\varphi>n)   \| v_X \|_{L^\infty(\mu_X)}  \| w_X \|_{\infty}^*).
\end{cases}
\]
~\end{lemma}
\begin{proof} By Lemma~\ref{lemma-of1} (the statement on $B$),
\[
D_A(z)(v_\Delta, w_\Delta)=\frac{1}{\bar\varphi}\int_\Delta v_\Delta \, d\mu_\Delta\int_\Delta \frac{A(z)-A(1)}{1-z} 1_Y w_\Delta \, d\mu_\Delta.
\]
By Lemma~\ref{lemma-Az-A1},  $\int_\Delta \frac{A(z)-A(1)}{1-z} 1_Y w_\Delta \, d\mu_\Delta=\sum_{n\geq 0} a_n z^n$
with $|a_n|=O(\mu_0(\varphi>n)  \| w_X \|_{\infty}^*)$. The statement on $|D_{A,n}(v_\Delta, w_\Delta)|$ follows.

Next, by definition,
\begin{align*}
Q(z)B(1) v_\Delta &= \frac{1}{\bar\varphi^2}\int_\Delta B(1) v_\Delta \, d\mu_\Delta\times (1-z) \sum_{n=0}^\infty \sum_{k=n+1}^\infty\mu_0(\varphi>k) z^n\\
&=\frac{1}{\bar\varphi^2}\int_\Delta v_\Delta \, d\mu_\Delta\times \sum_{n=0}^\infty\mu_0(\varphi>n) z^n.
\end{align*}
Thus,
\[
F_A(z)(v_\Delta, w_\Delta)=\frac{1}{\bar\varphi^2}\int_\Delta v_\Delta \, d\mu_\Delta\times \sum_{n=0}^\infty\mu_0(\varphi>n) z^n\times
\int_\Delta \frac{A(z)-A(1)}{1-z}  1_Y w_\Delta \, d\mu_\Delta.
\]
We already know that  the coefficients of $\int_\Delta \frac{A(z)-A(1)}{1-z} 1_Y w_\Delta \, d\mu_\Delta$
are $O(\mu_0(\varphi>n)  \| w_X \|_{\infty}^*)$. The statement on $ |F_{A,n}(v_\Delta, w_\Delta)|$ follows.
\end{proof}

The next result provides estimates for the coefficients of  $I^{B}(z)(v_\Delta, w_\Delta)$, $z\in\D$ defined in~\eqref{eq-decomp_main}. 
Write $I^B(z)=D_B(z)+F_B(z)$, where given that $Q(z)$ is as defined in~\eqref{eq-Q},
\[
D_B(z):=\int_\Delta A(1)\frac{P}{\bar\varphi}\frac{B(z)-B(1)}{1-z}  v_\Delta w_\Delta  \, d\mu_\Delta
\]
and
\[
F_B(z):=\int_\Delta A(1)Q(z)\frac{B(z)-B(1)}{1-z}  v_\Delta w_\Delta  \, d\mu_\Delta.
\]
\begin{lemma}\label{lemma-lastmain}Assume the setting of  Lemma~\ref{lemma-thirdmain}. 
Then the coefficients $D_{B,n}(v_\Delta, w_\Delta)$, $F_{B,n}(v_\Delta, w_\Delta)$ of the functions $D_B(z)(v_\Delta, w_\Delta)$, $F_B(z)(v_\Delta, w_\Delta)$, $z\in\D$
satisfy
\[
\begin{cases}
|D_{B,n}(v_\Delta, w_\Delta)|= O(\mu_0(\varphi>n)   \| v_X  \|_\theta^* \|w_X \|_{\infty}^*),\\[1mm] 
\, |F_{B,n}(v_\Delta, w_\Delta)|= O(\mu_0(\varphi>n)   \| v_X \|_{\theta}^* \| w_X \|_{\infty}^*).
\end{cases}
\]
~\end{lemma}

\begin{proof}The required argument is similar to the one used in the proof of  Lemma~\ref{lemma-thirdmain}
with Lemma~\ref{lemma-Vz-V1} replacing  Lemma~\ref{lemma-Az-A1}.~\end{proof}

\section{Proofs of Lemmas used in the proof of Theorem~\ref{prop-corel-infinite} (infinite measure case)} 
\label{sec-inf-estimates}

Recall that for $z\in\D$, the functions $I_{\inf}^A(z)(v_\Delta, w_\Delta):=\int_\Delta (A(z)-A(1)) T^*(z) B(1) v_\Delta w_\Delta  \, d\mu_\Delta$
and  $I_{\inf}^B(z)(v_\Delta, w_\Delta):=\int_\Delta A(1)T^*(z) (B(z)-B(1)) v_\Delta w_\Delta  \, d\mu_\Delta$
were defined in equation~\eqref{eq-decomp1}. Let 
 $I_{\inf,n}^{A}(v_\Delta, w_\Delta)$ and
$I_{\inf,n}^{B}(v_\Delta, w_\Delta)$ denote their $n$-th coefficients.
The next result provides estimates for $I_{\inf,n}^{A}(v_\Delta, w_\Delta)$ and 
$I_{\inf,n}^{B}(v_\Delta, w_\Delta)$ and it was used in the proof of Theorem~\ref{prop-corel-infinite}. 

\begin{lemma}~\label{lemma-IAn-IBn}
Assume (H0) b) and (H1). Let $v_X, w_X:X\to\R$ such that
$ \| v_X \|_\theta^*, \|w_X\|_{\infty}^*<\infty$. Then
\[
|I_{\inf,n}^{A}(v_\Delta, w_\Delta)|=O( n^{-\beta} \| v_X \|_\theta^* \|w_X\|_{\infty}^*), \quad |I_{\inf,n}^{B}(v_\Delta, w_\Delta)|=O(n^{-\beta}\| v_X \|_\theta^* \cdot\|w_X\|_{\infty}).
\]
\end{lemma}
The proof of the above result relies on  standard continuity properties  of the functions $I_{\inf}^A(z)$ and $I_{\inf}^B(z)$, $z\in\D$
which we recall below.

\subsection{Continuity properties of $I_{\inf}^A(z)$ and $I_{\inf}^B(z)$, $z\in\D$}

First we note some standard consequences  of Lemmas~\ref{lemma-An}, ~\ref{lemma-Bn2} and ~\ref{lemma-an}
which give the continuity properties of some quantities involving $A(z), B(z)$, $z\in\D$.

\begin{lemma}\label{lemma-an} Let $a(z)$ be a function acting on some function space $\B$ with norm $\|\,\|$, well defined on $\D$.
Suppose that its coefficients
satisfy  $\sum_{j>n}\|a_j\|\leq C_1 n^{-\beta}$ for $\beta>0$ and $C_1>0$.  
Then there exists $C_2>0$ such that for all $h>0$, all $u>0$ and all $\theta\in (-\pi,\pi]$,  $\|a(e^{-(u-i(\theta+h))}) - a(e^{-(u-i\theta)})\|\leq C_2 h^\beta$.~\end{lemma}

\begin{proof}This  proof is standard. We provide it here only for completeness. Compute that
\begin{equation}\label{eq-a}
\|a(e^{-(u-i(\theta+h))})-a(e^{-(u-i\theta)})\|\leq h\sum_{j\leq h^{-1}}j\|a_j\|+\sum_{j> h^{-1}}\|a_j\|.
\end{equation}
By assumption, the second term is  bounded by $C_1 h^\beta$.
Next, let $s_n:=\sum_{j>n}\|a_j\|$ and note that 
\begin{align*}
\sum_{j\leq h^{-1}}j\|a_j\|=\sum_{j\leq h^{-1}}j(s_{j-1}-s_j)&=\sum_{j\leq h^{-1}}(j-1)s_{j-1}-\sum_{j\leq h^{-1}}js_{j}+\sum_{j\leq h^{-1}}s_{j-1}\\
&\leq C_1 (h^{-1}-1) h^\beta +C_1h^{\beta-1}\leq 2C_1 h^{\beta-1}.
\end{align*}
Hence, the first term of~\eqref{eq-a} is bounded by  $2C_1 h^\beta$, as required.~\end{proof}

\begin{lemma}\label{lemma-Az-A1_2}
Assume (H0) (either a) or b)) and (H1). Let $v_Y:Y\to\R$, $w_X:X\to\R$ such that $\|v_Y\|_{L^\infty(\mu_0)}$, $\|w_X\|_{\infty}^*<\infty$. 
Then there exist $C_1, C_2>0$ such that for all $h>0$, all $u>0$ and all $\theta\in (-\pi,\pi]$,  
\begin{align*}
& \Big|\int_\Delta (A(e^{-(u-i(\theta+h))})-A(e^{-(u-i\theta)}))v_Y  w_\Delta \, d\mu_\Delta\Big| 
 \leq C_1 h^\beta \|v_Y\|_{L^\infty(\mu_0)}\| w_X \|_{\infty}^*,
\\[1mm]
& \Big|\int_\Delta (A(e^{-(u-i\theta)})-A(1))v_Y  w_\Delta \, d\mu_\Delta\Big|
\leq C_2 |\theta|^\beta\|v_Y\|_{L^\infty(\mu_0)}\|w_X \|_{\infty}^*.
\end{align*}
\end{lemma}

\begin{proof} By Lemma~\ref{lemma-An}, there exists $C>0$ such that
$\sum_{j=n}^\infty |\int_Y A_j  v_Y  w_\Delta\, d\mu_\Delta |\leq C \mu_0(\varphi>n)\|v_Y\|_{L^\infty(\mu_0)}
\| w_X \|_{\infty}^*$, for all $n\geq 0$. The conclusion follows by Lemma~\ref{lemma-an}.
\end{proof}

\begin{lemma}\label{lemma-Bz-B1_2}
Assume (H0) (either a) or b)) and (H1). Let $w_X:X\to\R$ such that  $\| v_X \|_\theta^* < \infty$. 
Then there exist $C_1, C_2>0$ such that  for all $h>0$, all $u\geq 0$ and all $\theta\in (-\pi,\pi]$, 
\begin{align*}
& 
\|(B(e^{-(u-i(\theta+h))})-B(e^{-(u-i\theta)}))v_\Delta\|_\theta\leq C_1 h^\beta\| v_X \|_\theta^*,\\[1mm]
&
\|(B(e^{-(u-i\theta)})-B(1))v_\Delta\| \leq C_2 |\theta|^\beta\| v_X \|_\theta^*.
\end{align*}
\end{lemma}

\begin{proof} By Lemma~\ref{lemma-Bn2}, there exists $C>0$ such that
 $\sum_{j\geq n}\|B_j v_\Delta\|\leq C\| v_X \|_\theta^*$,  for all $n\geq 0$. The conclusion follows from Lemma~\ref{lemma-an}.~\end{proof}

The following result was obtained in~\cite[Lemma 4.1]{MT} (see also~\cite[Lemma 2.4]{MT11} and its proof for a different argument).

\begin{lemma}\label{lemma-T*} Assume that  $F$  is Gibbs Markov and (H0)  b) holds. Then
for all  $u\geq 0$ and $\theta\in (-\pi,\pi]$,  there exist $C_1, C_2>0$ such that 
$\|T^*(e^{-(u-i\theta)})\|_\theta\leq C_1|u-i\theta|^{-\beta}$.  Moreover, for all  $h>0$, all $u\geq 0$ and all $\theta\in (-\pi,\pi]$, 
$\|T^*(e^{-(u-i(\theta+h))})-T^*(e^{-(u-i\theta)})\|_\theta \leq C_2 h^{\beta}|u-i\theta|^{-2\beta}$.
\end{lemma}

Combining Lemmas~\ref{lemma-Az-A1_2}, ~\ref{lemma-Bz-B1_2} and~\ref{lemma-T*}, we obtain

\begin{cor}\label{cor-IaIBinf}There exist positive constants $C_1, C_2$ such that
for all  $u\geq 0$ and all $\theta\in (-\pi,\pi]$, 
$|I_{\inf}^A(e^{-(u-i\theta)})(v_\Delta, w_\Delta)|\leq C_1\| v_X \|_\theta^*\| w_X \|_{L^\infty(\mu_X)}$ and similarly,
$|I_{\inf}^B(e^{-(u-i\theta)})(v_\Delta, w_\Delta)|\leq C_2 \| v_X \|_\theta^*\| w_X \|_{L^\infty(\mu_X)}$.
Moreover, there exist positive constants $C_3, C_4$ such that for all  $h>0$, all $u\geq 0$ and $\theta\in (-\pi,\pi]$,  
\begin{align*}
&|(I_{\inf}^A(e^{-(u-i(\theta+h)})-I_{\inf}^A(e^{-(u-i\theta)}))(v_\Delta, w_\Delta)|\leq C_3 h^{\beta} |u-i\theta|^{-\beta} \| v_X \|_\theta^*\|w_X\|_{L^\infty(\mu_X)}\\
& |(I_{\inf}^B(e^{-(u-i(\theta+h)})-I_{\inf}^B(e^{-(u-i\theta)}))(v_\Delta, w_\Delta)|\leq C_4 h^{\beta} |u-i\theta|^{-\beta} \| v_X \|_\theta^*\|w_X\|_{L^\infty(\mu_X)}
\end{align*}
\end{cor}

\subsection{Proof of Lemma~\ref{lemma-IAn-IBn}}

The first result below will be instrumental in the proof of Lemma~\ref{lemma-IAn-IBn}.

\begin{lemma}\label{lemma-b}
 Let $b(z)$ be a function well defined on $\D$.
Assume that there exist $C_1, C_2>0$ such that
for any $h>0$ and for all $\theta\in (-\pi,\pi]$,  $|b(e^{-(u-i\theta)})|\leq C_1$ 
 and $|b(e^{-(u-i(\theta+h))})-b(e^{-(u-i\theta)})|\leq C_2 h^\beta|u-i\theta|^{-\beta}$, for $\beta\in (0,1)$.
Then the $n$-th coefficient $b_n$ of $b(z)$, $z\in \D$ is $O(n^{-\beta})$.
\end{lemma}

\begin{proof}
We give the standard short proof only for completeness. We estimate the coefficients of $b(z)$, $z\in\D$, on the circle
 $\Gamma=\{e^{-u}e^{i\theta}:-\pi\leq \theta<\pi\}$ with $e^{-u}=e^{-1/n}$, where $n\geq 1$. Write
 \begin{align*}
b_n=\frac{1}{2\pi i}\int_\Gamma \frac{b(z)}{z^{n+1}} dz=
\frac{e}{2\pi}\int_{-\pi}^{\pi} b(e^{-1/n}e^{i\theta})e^{-in\theta}d\theta.
\end{align*}
Note that 
\[
|b_n|\ll \Big|\int_0^{1/n} b(e^{-1/n}e^{i\theta}) e^{-in\theta}\,d\theta\Big|+\Big|\int_{1/n}^{\pi} b(e^{-1/n}e^{i\theta})e^{-in\theta}\,d\theta\Big|.
\]
Since $| b(e^{-1/n}e^{i\theta})|\leq C_1$ we have $|\int_0^{1/n} b(e^{-1/n}e^{i\theta}) e^{-in\theta}\,d\theta|\leq C_1 n^{-1}$. To estimate the second term,
let $I:=\int_{1/n}^{\pi} b(e^{-1/n}e^{i\theta}) e^{-in\theta}\,d\theta$ and  note that
\[
I=\int_{1/n}^\pi  b(e^{-1/n}e^{i\theta})e^{-in\theta}\,d\theta
=-\int_{(1+\pi)/n}^{\pi+\pi/n} b(e^{-1/n}e^{i\theta}) e^{-in\theta}\,d\theta.
\]
Thus,
\[
2I=\int_{1/n}^\pi  b(e^{-1/n}e^{i\theta}) e^{-in\theta}\,d\theta
-\int_{(1+\pi)/n}^{\pi+\pi/n}b( e^{-1/n}e^{i(\theta-\pi/n)}) e^{-in\theta}\,d\theta=I_1+I_2-I_3,
\]
where
\begin{align*}
I_1  & = \int_\pi^{\pi+\pi/n} b(e^{-1/n}e^{i\theta})e^{-in\theta}\,d\theta, \qquad
I_2  = \int_{1/n}^{(1+\pi)/n} b(e^{-1/n}e^{i\theta}) e^{-in\theta}\,d\theta, \\
I_3 & =\int_{(1+\pi)/n}^\pi (b( e^{-1/n}e^{i(\theta-\pi/n)})-b(e^{-1/n}e^{i\theta}))e^{-in\theta}\,d\theta.
\end{align*}
Clearly, $|I_1|\ll n^{-1}$ and $|I_2|\ll n^{-1}$. By assumption, $|b( e^{-1/n}e^{i(\theta-\pi/n)})-b(e^{-1/n}e^{i\theta})|\leq C_2 n^{-\beta}|1/n-i\theta|^{-\beta}$.
Thus,  $|I_3|\ll n^{-\beta}$ and the conclusion follows.
\end{proof}

We can now complete
\begin{proof}[Proof of Lemma~\ref{lemma-IAn-IBn}] 
By Corollary~\ref{cor-IaIBinf},  $I_{\inf}^A(e^{-(u-i\theta)})(v_\Delta, w_\Delta)$ and $I_{\inf}^B(e^{-(u-i\theta)})(v_\Delta, w_\Delta)$
satisfy (in $|\, |$) the assumptions of Lemma~\ref{lemma-b} (where the involved constants include the product $\| v_X \|_\theta^*\|w_X\|_{\infty}$ ).
The conclusion follows by applying Lemma~\ref{lemma-b} to $I_{\inf}^A(e^{-(u-i\theta)})(v_\Delta, w_\Delta)$ and $I_{\inf}^B(e^{-(u-i\theta)})(v_\Delta, w_\Delta)$.
\end{proof}

\section{Non-Markov interval maps with indifferent fixed points}
\label{sec-AFN}

The works~\cite{Zweimuller98,Zweimuller00} studied a class
of non-Markov interval maps $f:[0,1] \to [0,1]$,
with indifferent fixed points, called
AFN maps, which stands for {\bf F}inite image, {\bf N}on-uniformly expanding
maps satisfying {\bf A}dler's distortion condition:
$f''/(f')^2$ is bounded.

\subsection{Known results for $f$ via first return inducing}
\label{sec-knownres}
For infinite measure preserving topologically mixing AFN maps $(f, [0,1],\mu)$, with $\mu_\tau(\tau>n)=n^{-\beta}\ell(n)$ for $\beta\in (1/2,1)$
and $\ell$ a slowly varying function,
and transfer operator $L$,~\cite[Theorem 1.1]{MT} shows that  $\lim_{n\to\infty}\ell(n)n^{1-\beta}L^n v=\frac{\sin\pi\beta}{\pi} \int v\, d\mu$,
uniformly on compact subsets of
$[0,1]\setminus I_p$, where $I_p$ is the set of indifferent fixed points,
for all $v=u/h$,  $u$ is a Riemann integrable on $[0,1]$ and  $h(x)=\frac{d\mu(x)}{dx}$. In particular, 
~\cite[Theorem 1.1]{MT} holds in the setting of \eqref{eq:f-neutral} below, for $v(x)=x^q$ with $q\beta\geq 1$.
For the LSV family of maps studied in~\cite{LiveraniSaussolVaienti99}, which induce with first return to a Gibbs Markov map, 
the work~\cite{MT} also obtains higher order asymptotics of $L^n v$, for some suitable $v$ supported on $(0,1]$ (we recall
 that such a map has a single indifferent fixed point at $0$); this result of~\cite{MT} implies higher order asymptotics for
the correlation function $\rho_n(v,w)=\int v w\circ f^n\,d\mu$ associated with the LSV family of
 maps~\cite{LiveraniSaussolVaienti99} (for the suitable $v$ and $w\in L^\infty$). These results on higher order asymptotics
have been improved in~\cite{Terhesiu12} and again, they apply to LSV maps~\cite{LiveraniSaussolVaienti99}.
Higher order
asymptotics of $\rho_n(v,w)$ in the setting of AFN maps without
Markov partition has \emph{ not} been addressed. The only obstacle in~\cite{MT, Terhesiu12} was that
the invariant density of the induced map is $BV$ and thus, the arguments used in~\cite{MT, Terhesiu12} to obtain higher order expansion
of $\mu_\tau(\tau>n)$ (which require smoothness of the induced invariant density)
do not apply\footnote{Higher order  expansion of $\mu(\tau>n)$
is required for results  which aim to address any type of error term in the infinite measure set-up: see~\cite{MT, Terhesiu12, Terhesiu12b}.}.

In what follows, in the process of verifying (H0) and (H1) for AFN maps, we obtain excellent estimates for $\mu_\tau(\tau>n)$.
This allows one to infer that the results in~\cite{MT, Terhesiu12} on higher order asymptotics of $L^n$ hold in the setting
of \eqref{eq:f-neutral}, a typical examples in the class
of AFN maps~\cite{Zweimuller98,Zweimuller00}; we recall that Theorems~\ref{prop-corel-finite}  and
\ref{prop-corel-infinite} 
only address the asymptotics of the correlation function $\rho_n(v,w)$
for appropriate $v,w$ (so a weaker result than higher order asymptotics of $L^n$). For details we refer to Section~\ref{sec-moregen}.

In the setting of finite measure preserving non-Markov, non-uniformly expanding interval maps $f:[0,1] \to [0,1]$
with a single indifferent fixed point at $0$, the works \cite{MT13, HuVaienti} consider
a first return induced map to $Y = [z,1]$, $z > 0$,
 to obtain upper/sharp mixing rates. The relevant Banach space in which renewal type arguments are developed or verified is 
$BV$.  The sharp results in \cite{HuVaienti} are for observables supported on $Y$.

\subsection{Verifying conditions (H0)-(H2)}\label{sec-ver}
One can verify the abstract conditions in Section~\ref{sec-frame},
and hence prove Theorems~\ref{prop-corel-finite}  and
\ref{prop-corel-infinite} for the general class of 
AFN maps  studied~\cite{Zweimuller98,Zweimuller00}.
For simplicity, we restrict here to the following example:
\begin{equation}\label{eq:f-neutral}
f(x) = f_{\alpha,b}(x) = x(1+ b x^\alpha) \bmod 1, \qquad
\alpha > 0,\, b \in (0,1].
\end{equation}
We induce on the interval $Y= [e_0,1]$, where $e_0 \in (0,1)$ is such that
$f(e_0) = 0$.
The fact that the orbit of $e_0$ is disjoint from the interior of $Y$ implies that
$e_0 \not\in f^i(a)$ for every $a \in \alpha$, $0 \leq i < \varphi(a)$, and therefore
condition (H2) follows immediately.

Adler's condition fails at $x = 0$ if $\alpha \in (0,1)$ 
in~\eqref{eq:f-neutral}, but the first return map $f^\tau$ to $Y$ 
is uniformly expanding and Adler's condition does hold for it.
This gives a uniform bound on the distortion of $g:=f^\tau$.
Indeed, if $g:J \to g(J)$ is a branch of $f^\tau$ with
$|g''(s)/(g'(s))^2| \leq C$, then for all $x,y\in J$,
\begin{eqnarray}\label{eq:distortion}
\left|\log \frac{g'(y)}{g'(x)}\right| &=& \left|\int_x^y \frac{d}{ds} \log g'(s) \, ds\right|  \nonumber\\
&=& \int_x^y \left|\frac{g''(s)}{g'(s)}\right| \, ds
\leq  C\int_x^y |g'(s)| \, ds = C |g(y)-g(x)|.
\end{eqnarray}
The same bound applies to iterates of $g$.
As a consequence, the proportion of subintervals of
the branch domains of $g^k$ doesn't vary too much under the map $g^k$.
This fact will be used throughout this section.

In general, $f$ is not Markov, but preserves an absolutely
continuous measure which is finite if and only if $\alpha \in (0,1)$.
Set $\beta = 1/\alpha$.
Let $x_0 = e_0$ and for $n \geq 1$, define recursively $x_{n+1} < x_n$ 
so that $f(x_n) = x_{n-1}$.
From \cite{Holland05} (in fact, sharper estimates can be found in 
\cite[Section B]{Terhesiu12}) one can establish the asymptotics
\begin{equation}\label{eq:asymp}
x_n = \frac{c^*}{(n+1)^{\beta}} + O\left(\frac{\log(n+1)}{(n+1)^{\beta+1}} \right)
\qquad  \text{ for some } c^* = c^*(\alpha) > 0.
\end{equation} 
For instance the condition $\| v_X \|_\infty^* < \infty$
can thus be written as $\sup_{x\in (0,1]} x^{-(1+\eps)/\beta } |v_X(x)| < \infty$.

For each $k \geq 1$, let $e_k > e_{k-1}$ be the right-most point
such that $f^{\tau_k(e_k)}(e_k) = e_0$. Then $f^{\tau_k}$ maps $[e_k, 1)$
monotonically but in general not surjectively into $Y$.
The general return time is
$\varphi(y) = \tau_k(y) + \tau \circ f^{\tau_k(y)}(y)$ for $y \in [e_k, e_{k+1})$.
The map $f^\varphi$ has only onto branches and thus, it is a Gibbs Markov induced map with good distortion 
properties derived from \eqref{eq:distortion}, and $\{ y \in Y : \varphi(y) = \tau_{k+1}(y) \} = [e_k, e_{k+1})$.

\begin{lemma}\label{lem:H0}
The family of maps~\eqref{eq:f-neutral} satisfy condition (H0).
\end{lemma}

\begin{proof}
For $n \geq 1$, let $A_k := \{ y \in Y : \varphi(y) = \tau_{k+1}(y) > n\} \subset [e_k, e_{k+1})$.
Let in the remainder of this section $\tau_k = \tau_k(1)$.
We first estimate the derivatives $\lambda_k := Df^{\tau_k+1}(e_k)$ and lengths $|A_k|$.

For $j \geq 1$, let $y_j \in [e_0, e_1)$ be such that $f(y_j) = x_{j-1}$. Hence
$\tau(y_j) = j$ and $f^j(y_j) = e_0$,
so that $\{ \tau > n\} = (e_0, y_n)$.
Let $\sigma_j$ be the integers such that $f^{\tau_j}(1) \in [y_{\sigma_{j+1}},
 y_{\sigma_{j+1}-1})$ for $j \geq 0$, see Figure~\ref{fig:AFN}.
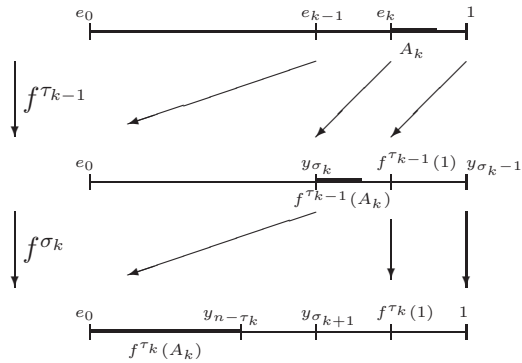
\begin{figure}[h]
\begin{center}
\unitlength=10mm
\begin{picture}(10,5)(2,-0.3)
\put(5,4){\line(1,0){5}} 
\put(9,4.03){\line(1,0){0.6}}\put(9.1,3.7){\tiny $A_k$}
\put(5,3.9){\line(0,1){0.2}} \put(4.8,4.2){\tiny $e_0$}
\put(8,3.9){\line(0,1){0.2}} \put(7.8,4.2){\tiny $e_{k-1}$}
\put(9,3.9){\line(0,1){0.2}} \put(8.8,4.2){\tiny $e_k$}
\put(10,3.9){\line(0,1){0.2}} \put(10,4.2){\tiny $1$}
\put(8,3.6){\vector(-3,-1){2.5}}
\put(9,3.6){\vector(-1,-1){1}}
\put(10,3.6){\vector(-1,-1){1}}
\put(4,3.6){\vector(0,-1){1}}\put(4.1,3){$f^{\tau_{k-1}}$}
\put(5,2){\line(1,0){5}} 
\put(8.3,2.03){\line(1,0){0.3}}
\put(8,2.03){\line(1,0){0.5}}\put(7.7,1.7){\tiny $f^{\tau_{k-1}}(A_k)$}
\put(5,1.9){\line(0,1){0.2}} \put(4.8,2.2){\tiny $e_0$}
\put(8,1.9){\line(0,1){0.2}} \put(7.8,2.2){\tiny $y_{\sigma_k}$}
\put(9,1.9){\line(0,1){0.2}} \put(8.8,2.2){\tiny $f^{\tau_{k-1}}(1)$}
\put(10,1.9){\line(0,1){0.2}} \put(10,2.2){\tiny $y_{\sigma_k-1}$}
\put(8,1.6){\vector(-3,-1){2.5}}
\put(9,1.5){\vector(0,-1){0.8}}
\put(10,1.6){\vector(0,-1){1}}
\put(4,1.6){\vector(0,-1){1}}\put(4.1,1){$f^{\sigma_k}$}
\put(5,0){\line(1,0){5}} 
\put(5,0.03){\line(1,0){2}}\put(5.5,-0.3){\tiny $f^{\tau_k}(A_k)$}
\put(5,-0.1){\line(0,1){0.2}} \put(4.8,0.2){\tiny $e_0$}
\put(7,-0.1){\line(0,1){0.2}} \put(6.5,0.2){\tiny $y_{n-\tau_{k}}$}
\put(8,-0.1){\line(0,1){0.2}} \put(7.8,0.2){\tiny $y_{\sigma_{k+1}}$}
\put(9,-0.1){\line(0,1){0.2}} \put(8.8,0.2){\tiny $f^{\tau_k}(1)$}
\put(10,-0.1){\line(0,1){0.2}} \put(9.9,0.2){\tiny $1$}
\end{picture}
\caption{\label{fig:AFN} The points $e_{k-1}, e_k, 1$ and set $A_k$ 
and their images.}
\end{center}
\end{figure}
Then also $f^{\tau_{k-1}}(A_k) \subset f^{\tau_{k-1}}([e_k, 1) ) \subset 
[y_{\sigma_k},  y_{\sigma_k-1})$ for each $k \geq 0$.
Using \eqref{eq:distortion} (or the fact that 
all branches of $f^\tau$ are convex upwards), we find
$$
\frac{|1-e_{j+1}|}{|1-e_j|} \le
\frac{|f^{\tau_j}(1)-y_{\sigma_{j+1}}|}{|f^{\tau_j}(1) - e_0|} \le
\frac{|y_{\sigma_{j+1}-1} - y_{\sigma_{j+1}}|}{|y_{\sigma_{j+1}} - e_0|}
\frac{|f^{\tau_j}(1)-y_{\sigma_{j+1}}|}{|y_{\sigma_{j+1}-1} - y_{\sigma_{j+1}}|}.
$$
Using \eqref{eq:asymp}, we can bound 
$\frac{|y_{\sigma_{j+1}-1} - y_{\sigma_{j+1}}|}{|y_{\sigma_{j+1}} - e_0|} \leq \min( \gamma/\sigma_{j+1}, 1/\lambda)$
for some uniform $\gamma > 0, \lambda > 1$.
We bound the second factor $\frac{|f^{\tau_j}(1)-y_{\sigma_{j+1}}|}{|y_{\sigma_{j+1}-1} - y_{\sigma_{j+1}}|} \leq 1$, except for $j = k-1$.
Taking the product over $j = 0,\dots, k-1$ gives
\begin{equation}\label{eq:p1}
\frac{|1-e_k|}{|1-e_0|} \le
\frac{|f^{\tau_{k-1}}(1)-y_{\sigma_k}|}{|y_{\sigma_k-1} - y_{\sigma_k}|}
\prod_{j=1}^k \min\{ \frac{\gamma}{\sigma_j} , \frac1\lambda \}.
\end{equation}
Boundedness of distortion of
$f^{\sigma_k}:[y_{\sigma_k} , y_{\sigma_k-1}) \to [e_0,1)$ together with \eqref{eq:asymp} implies
\begin{equation}\label{eq:p2}
\frac{|f^{\tau_{k-1}}(1)-y_{\sigma_k}|}{|y_{\sigma_k-1} - y_{\sigma_k}|}
\ll \frac{|f^{\tau_k}(1)-e_0|}{|1-e_0|} \le
\frac{c^*}{|1-e_0|} \sigma_{k+1}^{-\beta} (1+o(1)).
\end{equation}
The boundedness of distortion of $f^{\tau_k+1}: [e_k,1] \to [0, f^{\tau_k+1}(1)]$
combined with \eqref{eq:p1} and  the first inequality of \eqref{eq:p2}
leads to
\begin{align*}
\frac{1}{\lambda_k} &\ll \frac{1}{f'(e_0)} \frac{|1-e_k|}{|f^{\tau_k}(1)-e_0|}\\
&\ll
\frac{|f^{\tau_{k-1}}(1)-y_{\sigma_k}|}{|y_{\sigma_k-1} - y_{\sigma_k}|}
\frac{|1-e_0|}{|f^{\tau_k}(1)-e_0|}
\prod_{j=1}^k \min\{ \frac{\gamma}{\sigma_j} , \frac1\lambda \}  
\ll \prod_{j=1}^k  \frac{1}{ \max\{ \sigma_j/\gamma , \lambda \}}.
\end{align*}
As $\tau_k = \sum_{j=1}^k \sigma_j$, the quotients
$\tau_k/\lambda_k$ decrease exponentially and
are thus summable in $k$.

Define $h_{\max} := \sup_{x\in [e_0,1]} h(x)$ for
the density $h = \frac{d\mu_0}{dx}$. Take $n_0$ so large that  
\begin{equation}\label{eq:n0}
\prod_{j=1}^{k_0}  \frac{1}{ \max\{ \sigma_j/\gamma , \lambda \}} <  \frac1{n^{\beta+1}}
\quad \text{ and } \quad
c^* \sum_{k = k_0}^\infty \frac{h_{\max} }{\lambda_k} <  \frac1{n^{\beta+1}}
\end{equation}
for all $n \geq n_0$, where $k_0 = k_0(n) = \min\{k \geq 1 : 2\tau_{k+2} \geq n\}$.

Let $g_k$ denote the inverse branch of $f^{\tau_k+1}: A_k \to [0, x_{n-\tau_k-1}]$.
With the notation $\lambda_k = Df^{\tau_k+1}(e_k)$,
$\lambda'_k = D^2f^{\tau_k+1}(e_k)$ and
$q_k = -\lambda'_k/\lambda_k^2$ (which is bounded in $k$ because of
Adler's condition), we find
$$
g_k(x)-e_k = \frac1{\lambda_k}\left(x + \frac{q_k}{2} x^2 + O(x^3) \right).
$$
Since the density $h$ is $C^2$ smooth,
$\mu_0(A_k) = h(e_k) (g_k(x_{n-\tau_k}) - e_k) + \frac{h'(e_k)}{2} (g_k(x_{n-\tau_k}) - e_k)^2+ O((g_k(x_{n-\tau_k}) - e_k)^3)$.
Inserting the asymptotics for $g_k(x)-e_k$ and for
$x_{n-\tau_k-1}$ from \eqref{eq:asymp}, gives
\begin{eqnarray*}
\mu_0(A_k) &=& \frac{h(e_k)c^*}{\lambda_k} (n-\tau_k)^{-\beta}
+  \frac{(c^*)^2}{2\lambda_k} \Big(h(e_k) q_k + \frac{h'(e_k)}{\lambda_k}\Big) (n-\tau_k)^{-2\beta}\\
&& \quad +\ O\left( \frac{h(e_k)}{\lambda_k} (n-\tau_k)^{-(\beta+1)}\log(n-\tau_k) \right).
\end{eqnarray*}
Applying this for $k < k_0$ and therefore $n > 2\tau_k$, we can use the asymptotics $(n-\tau_k)^{-\beta} \leq n^{-\beta}(1+2\beta\tau_k/n)$. Therefore
\begin{eqnarray*}
\mu_0(A_k) &=&  n^{-\beta} \frac{h(e_k)c^*}{\lambda_k} 
+ O \left(\frac{\tau_k}{\lambda_k} n^{-(\beta+1)} +  n^{-(\beta+1)} \log(n)
+ n^{-2\beta} \right).
\end{eqnarray*}
Set $c := c^* \sum_{k \geq 0} \frac{ h(e_k) }{\lambda_k}$.
Because $\sum_{k \geq 0} \frac{\tau_k}{\lambda_k} < \infty$
and $|1-e_{k_0}| \leq |1-e_0| \prod_{j=1}^{k_0} \frac{1}{\max\{ \sigma_j/\gamma , \lambda \}}$ by \eqref{eq:p1} for $k = k_0$, we obtain
\begin{eqnarray*}
|\mu_0(\varphi > n) -  c n^{-\beta}| &=& \left| \sum_{k=0}^{k_0-1} \mu_0(A_k)
 + \mu_0([e_{k_0}, 1)) - \frac{c^*}{n^{\beta}} \sum_{k=1}^\infty \frac{h(e_k)}{\lambda_k} \right| \\
&=& \frac{c^*}{n^{\beta}} \sum_{k=k_0}^\infty \frac{h(e_k)}{\lambda_k} + 
\prod_{j=1}^{k_0} \frac{ h_{\max} }{ \max\{ \sigma_j/\gamma , \lambda \}} \\
&&\quad + \  O( n^{-(\beta+1)} \log n ) + O( n^{-(\beta+1)} ) +  O( n^{-2\beta} ).
\end{eqnarray*}
Recalling the choice of $n_0$ and hence $k_0$ in
\eqref{eq:n0}, we conclude that 
$|\mu_0(\varphi > n) -  c n^{-\beta}| = O(n^{-2\beta}, n^{-(\beta+1)} \log n)$, and condition (H0) follows.
\end{proof}

\begin{lemma}\label{lemma-compind}
For the maps given by \eqref{eq:f-neutral} we have
$\Big| \frac{1}{\bar\rho} \mu_0(\varphi > n) - \mu_\tau(\tau>n) \Big|
= O(n^{-(1+\beta)})$
and (H1) holds.
\end{lemma}

\begin{proof}
To estimate $| \frac{1}{\bar\rho} \mu_0(\varphi > n) - \mu_\tau(\tau>n)|$,
we use Lemma~\ref{lem:mutau}.
Recall that $\lambda_k = Df^{\tau_k+1}(e_k)$ and $h_{\max} = \sup_{x\in [e_0,1]} h(x)$.
Take again $n_0$ and $k_0 = k_0(n)$ as in \eqref{eq:n0} and 
assume that $n \geq n_0$, so $n > 2\tau_k$ for all $k < k_0$.

Recall that $\{ \varphi = \tau_{k+1}\} = [e_k, e_{k+1})$.
The set of $y \in [e_k, e_{k+1})$ such that $n-\tau_k < \tau(f^{t_k(y)}(y)) \leq n$
is $O\Big(\frac{|y_n-y_{n-\tau_k}|}{|y_{\sigma_{k+1}} - e_0|}\Big)$ due to
 the boundedness of distortion of
$f^{\tau_k}:[e_k,1] \to [e_0, f^{\tau_k}(1)]$. Using also
\eqref{eq:asymp} to estimate $|y_n-y_{n-\tau_k}|$, 
we obtain for the first sum in \eqref{eq:mutau1}:
$$
\sum_{k \geq 0} \int_{ \{ \varphi = \tau_{k+1}\} }  
1_{ \{ n \geq \tau > n-\tau_k\} } \circ f^{\tau_k} \, d\mu_0
\leq h_{\max}
\sum_{k \geq 1}  |y_n-y_{n-\tau_k}| \frac{ |e_{k+1}-e_k| }{ |y_{\sigma_{k+1}} - e_0| } .
$$
By boundedness of distortion, 
$\frac{ |e_{k+1}-e_k| }{ |y_{\sigma_{k+1}} - e_0| } \ll \frac{1}{\lambda_k}$ and
since $\frac{\tau_k}{\lambda_k}$ is summable by 
Lemma~\ref{lem:H0}, we conclude
\begin{align}\label{eq:gg}
\sum_{k \geq 0} \int_{ \{ \varphi = \tau_{k+1}\} } &
1_{ \{ n \geq \tau > n-\tau_k\} } \circ f^{\tau_k} \, d\mu_0
\ll h_{\max} \Big(
\sum_{k=1}^{k_0-1} \frac{(n-\tau_k)^{-\beta} - n^{-\beta}}{\lambda_k} 
+ \sum_{k\geq k_0} \frac{|1-e_0|}{\lambda_k} \Big) \nonumber \\
\le& \Big(  h_{\max}
\sum_{k = 1}^{k_0-1} \frac{2\beta \tau_k}{\lambda_k} \ + \ 
\frac{|1-e_0|}{c^*} \Big) n^{-(\beta+1)} \ll  n^{-(\beta+1)}.
\end{align}
From \eqref{eq:asymp} we derive
$\frac{| f^{\tau_k}(1) - y_{\sigma_{k+1}}|}{|f^{\tau_k}(1)-e_0|} \ll \frac1{\sigma_{k+1}}$.
Using \eqref{eq:p1},
the second sum in \eqref{eq:mutau1} can be estimated as 
\begin{eqnarray*}
\int_{\{ \varphi > \tau_{k+1}\}} 1_{\{\tau > n\}} \circ f^{\tau_k} \, d\mu_0 &=&
\sum_{\sigma_{k+1} > n} \mu_0([e_{k+1}, 1]) \\
&\le& h_{\max} \sum_{\sigma_{k+1} > n} 
\frac{| f^{\tau_k}(1) - y_{\sigma_{k+1}}|}{|f^{\tau_k}(1)-e_0|} 
\prod_{i=1}^k \frac{1}{\max (\sigma_j/\gamma, \lambda)} \\
&\ll&  \frac{h_{\max}}{n} \sum_{\sigma_{k+1} > n}
\prod_{j=0}^k \frac{1}{\max (\sigma_j/\gamma, \lambda)} \\
&\le& \frac{h_{\max}}{n} 
\prod_{j=1}^{k_0} \frac{1}{\max (\sigma_j/\gamma, \lambda)}
\sum_{k \geq k_0, \sigma_{k+1} > n}
\prod_{j=k_0+1}^k \frac{1}{\max (\sigma_j/\gamma, \lambda)} \\
\end{eqnarray*}  
because $\sigma_{k+1} \leq n$ for $k \leq k_0$
by the definition of $k_0 = k_0(n)$ in \eqref{eq:n0},
which also gives
$\prod_{j=1}^{k_0} \frac{1}{\max (\sigma_j/\gamma, \lambda)}
\leq n^{-(1+\beta)}$ for all $n \geq n_0$. 
Therefore the quantity of the previous displayed equation
is $O(n^{-(\beta+2)})$, which is clearly negligible compared to the first 
term \eqref{eq:gg}.

To check condition (H1), we continue the proof of Lemma~\ref{lem:H0}.
Boundedness of distortion of 
$f^{\sigma_{k+1}}:[y_{\sigma_{k+1}} , y_{\sigma_{k+1}-1}) \to [e_0,1)$ gives
$\frac{|A_k|}{|1-e_k|} \ll
\frac{|y_{n-\tau_k} - e_0|}{|y_{\sigma_{k+1}-1} - e_0|}
\ll \frac{ \sigma_{k+1}^\beta }{ (n-\tau_k)^\beta }$.
Combining this with \eqref{eq:p1} and \eqref{eq:p2} we get
\begin{equation}\label{eq:k}
(k+1) |A_k| \ll
 \frac{ k+1 }{ (n-\tau_k)^\beta }
 \prod_{j=1}^k \min\{ \frac{\gamma}{\sigma_j} , \frac1\lambda \} 
= \frac{1}{n^{\beta}} 
\left( \frac{ n }{(n-\tau_k)\tau_k} \right)^\beta
\frac{ (k+1) \tau_k^{\beta}}{\prod_{j=1}^k \max\{ \frac{\sigma_j}{\gamma} , \lambda \} }.
\end{equation}
Since $n \leq (n-\tau_k)\tau_k$, this gives
$$
\int_{\{\varphi > n\}} \rho(y) \, d\mu_0 \leq h_{\max}
\sum_{k \geq 0} (k+1) |A_k| \ll 
\frac{1}{n^\beta} \sum_{k \geq 0} \frac{(k+1) \tau_k^\beta}
{\prod_{j=1}^k \max\{  \frac{\sigma_j}{\gamma} , \lambda \} }.
$$
Recall that $\tau_k = \sum_{j=1}^k \sigma_j$, so the sum in this expression
is finite. Since  $n^{-\beta} = O(\mu_0(\varphi > n))$, condition (H1)
follows.
\end{proof}

\subsection{Further results for the infinite measure setting of
\eqref{eq:f-neutral}} \label{sec-moregen}

Recall that $\mu_\tau$ is the absolutely continuous probability measure  preserved by the first return map $f^\tau:Y \to Y$.
Lemma~\ref{lemma-compind} shows that the tails 
$\frac{1}{\bar\rho} \mu_0(\varphi > n)$ and 
$\mu_\tau(\tau > n)$ coincide up to $O(n^{-(1+\beta)})$. 
As shown in Lemma~\ref{lem:H0},
 $\mu_0(\varphi > n)$ satisfies (H0) b). Thus, $\mu_\tau(\tau > n)$ also satisfies (H0) b).
Moreover, using sharper estimates of $x_n$ (as 
in~\cite[Section B]{Terhesiu12}), one obtains sharper estimates for $\mu_0(\varphi > n)$; in particular,
$\mu_0(\varphi > n)$ satisfies condition (H) in~\cite{Terhesiu12},
and by Lemma~\ref{lemma-compind}, $\mu_\tau(\tau > n)$ satisfies 
condition (H) in~\cite{Terhesiu12} as well.

As mentioned in Section~\ref{sec-knownres}, the only obstruction in~\cite{MT, Terhesiu12} to obtain higher
order asymptotics of the transfer operator $L^n$ for maps such as~\eqref{eq:f-neutral}
uniformly on $(0,1]$, for $BV$ functions supported on $(0,1]$,
was the higher order expansion of $\mu_\tau(\tau > n)$. But as shown here $\mu_\tau(\tau > n)$ satisfies (H0) b) (by Lemmas~\ref{lem:H0} and~\ref{lemma-compind})
and hence, the required tail conditions of~\cite[Theorem 9.1, Theorem 11.4]{MT} and~\cite[Theorem 1.1, Proposition 1.6]{Terhesiu12}.
As a consequence, these results apply to the map~\eqref{eq:f-neutral}.

Using the fact that ~\cite[Theorem 11.4]{MT} and~\cite[Proposition 1.6]{Terhesiu12} hold for the map~\eqref{eq:f-neutral},
one also obtains~\cite[Corollary 9.10]{MT} and ~\cite[Proposition 1.7]{Terhesiu12}, which provide error rates
for the arcsine law. ( It is known that arcsine laws hold for the general class of AFN maps, see~\cite{ThalerZweimuller06}.)
As shown in~\cite{Zweimuller00, ThalerZweimuller06}, the Darling Kac law 
 holds for the general
class of AFN maps considered in~\cite{Zweimuller00}. Error rates in the Darling Kac law for maps such as
the one studied in~\cite{LiveraniSaussolVaienti99}, characterized by good higher order asymptotics 
of the tail of the first return time, were obtained in~\cite[Theorem 1.1]{Terhesiu12b}. Again, the only obstruction
in~\cite{Terhesiu12b} to show that~\cite[Theorem 1.1]{Terhesiu12b} applies to maps of the form~\eqref{eq:f-neutral},
was the lack of knowledge on the higher order expansion of $\mu_\tau(\tau > n)$. Given the information on
$\mu_\tau(\tau > n)$ obtained here, one obtains that~\cite[Theorem 1.1]{Terhesiu12b} applies to 
the setting of~\eqref{eq:f-neutral}.

\appendix
\section{Comparing general and first returns}

In this appendix we prove a result used in Lemma~\ref{lemma-compind}, namely Lemma~\ref{lem:mutau}. A consequence
of  Lemma~\ref{lemma-compind}, namely Corollary~\ref{cor:mutau}  below, allows for a direct comparison between
$\sum_{j \geq n} \frac{1}{\bar\varphi} \mu_0(\varphi > j)$ (the leading term of the correlation decay in Theorem~\ref{prop-corel-finite})
and $\sum_{j \geq n}\frac{1}{\bar\tau} \mu_\tau(\tau > j)$ (the leading term of the  correlation decay, possibly, obtained via inducing
with respect to first return time); we refer to Remark~\ref{rem:mutau} for details in the setting of \eqref{eq:f-neutral}.

\begin{lemma}\label{lem:mutau}
Suppose that $\mu_0$ and $\mu_\tau$ are equivalent
measures on $Y$ that are preserved by the general return
map $f^\varphi$ and first return map $f^\tau$, respectively, and $\bar \rho = \int_Y \rho\, d\mu_0 < \infty$. Then
\begin{align}\label{eq:mutau1}
\frac{1}{\bar\rho}   \mu_0(\varphi & > j) - 
\mu_\tau(\tau > j) \nonumber \\
=\ &
\frac{1}{\bar\rho} \sum_{k \geq 0} 
\left( \int_{\{ \varphi = \tau_{k+1}\}} 1_{\{j \geq \tau > j-\tau_k\}}  \circ f^{\tau_k}
d\mu_0 -
\int_{\{ \varphi > \tau_{k+1} \}} 1_{\{\tau > j\}}  \circ f^{\tau_k}
d\mu_0 \right).
\end{align}
\end{lemma}

\begin{proof}
The set $\hat Y = \pi^{-1}(Y) =
\sqcup_{k\geq 0} \hat Y_k = \sqcup_{k \geq 0} \{ (y, \tau_k) : \rho > k\}$ 
can be considered as a subtower of $\Delta$
with dynamics
$$
T_\Delta^\tau(y,\tau_i) = \begin{cases}
(y,\tau_{i+1}) & \text{ if } 0 \leq i < \rho(y)-1,\\
(Fy,0) & \text{ if } i = \rho(y)-1,
\end{cases}
$$
see Figure~\ref{fig:tower}.
Clearly $\mu_0$ is the invariant measure of the return map to the base.
Recall that $\mu_{\Delta}$ is the ``pushed-up'' measure from $\mu_0$
onto $\Delta$. Restricted to $\hat Y$,
$\mu_{\Delta}$ is $T_\Delta^\tau$-invariant, and
$\mu_{\Delta}(\hat Y) = \int_Y \rho \, d\mu_0 = \bar \rho < \infty$.
\begin{figure} [ht]
\begin{center}
\unitlength=5mm
\begin{picture}(12,6.5)
\thicklines
\put(11,0){\line(-1,0){12}} \put(11.5,-0.2){$\hat Y_0$}
\put(11,2){\line(-1,0){8}} \put(11.5,1.8){$\hat Y_1$}
\put(11,5){\line(-1,0){2}} \put(11.5,4.8){$\hat Y_2$}
\put(11,6){\line(-1,0){0.5}}\put(11.5,5.8){$\hat Y_3$}
\put(11,0.02){\line(-1,0){12}} 
\put(11,2.02){\line(-1,0){8}} 
\put(11,5.02){\line(-1,0){2}}
\put(11,6.02){\line(-1,0){0.5}}
 
\thinlines
 \put(-1,-0.2){$\underbrace{\qquad\qquad\quad}_{\varphi = \tau_1}$}
 \put(3,-0.2){$\underbrace{\qquad\qquad\qquad\quad}_{\varphi = \tau_2}$}
\put(8.8,-0.2){$\underbrace{\qquad}_{\varphi = \tau_3}$}
\put(11,0.5){\line(-1,0){11}}
\put(11,1){\line(-1,0){10}} 
\put(11,1.5){\line(-1,0){9}}
\put(11,2.5){\line(-1,0){7}} 
\put(11,3){\line(-1,0){6}} 
\put(11,3.5){\line(-1,0){5}}
\put(11,4){\line(-1,0){4}} 
\put(11,4.5){\line(-1,0){3}} 
\put(11,5.5){\line(-1,0){1}}
\put(11,6.5){\line(-1,0){0.2}}
\end{picture}
\end{center}
\caption{\label{fig:tower} The tower $\Delta$ and $\hat Y$ in between
with the bold-face levels.}
\end{figure}
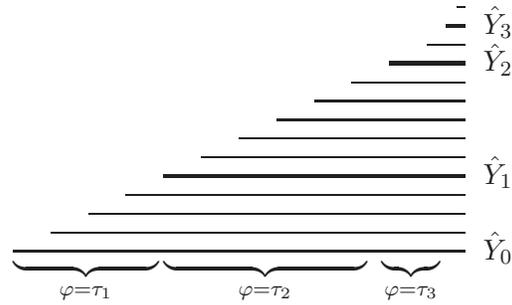 
The projection $\hat\pi:\hat Y \to Y$, $(y,k) \mapsto f^{\tau_k}(y)$
pushes $\mu_{\Delta}$ down to an $f^\tau$-invariant measure, which when normalized
has the formula
$\mu_\tau(A) = \frac{1}{\bar \rho} 
\sum_{k \geq 0} \mu_0(\hat \pi^{-1}(A) \cap \hat Y_k)$.
Applying this to $A = \{ \tau > n\}$ and recalling that 
$\hat Y_k = \{ \varphi > \tau_k\} \times \{ \tau_k \}$ gives
\begin{equation}\label{eq:mutau}
\mu_\tau(\tau > n) 
=  \frac{1}{\bar \rho} \sum_{k \geq 0} \int_{\{\varphi > \tau_k\}}
1_{\{ \tau > n\}} \circ f^{\tau_k}(y) \, d\mu_0(y).
\end{equation}
Next we specify the set $\{ \varphi > n\}$.
For each $y \in Y$, pick $k = \rho(y)-1$, so that $(y,k)$ is at the top level
of the subtower $\hat Y$, so $\varphi(y) = \tau_k(y) + \tau(f^{\tau_k}(y))$.
Therefore
\begin{eqnarray}\label{eq:muphi}
\mu_0(\varphi > n) &=&
\sum_{k \geq 0} \mu_0( \varphi = \tau_{k+1} \wedge \tau \circ f^{\tau_k} > n - \tau_k) \nonumber \\
&=& \sum_{k \geq 0} \int_{\{ \varphi = \tau_{k+1}\}}
1_{\{\tau > n-\tau_k\}} \circ f^{\tau_k} \, d\mu_0.
\end{eqnarray}
Combining \eqref{eq:mutau} and \eqref{eq:muphi} we get:
\begin{align*}
\frac{1}{\bar\rho}   \mu_0(\varphi & > j) - 
\mu_\tau(\tau > j) \nonumber \\
=\ &
\frac{1}{\bar\rho} \sum_{k \geq 0} 
\left( \int_{\{ \varphi = \tau_{k+1}\}} 1_{\{\tau > j-\tau_k\}}  \circ f^{\tau_k}
d\mu_0 -
\int_{ \{ \varphi > \tau_k\} } 1_{\{\tau > j\}  } \circ f^{\tau_k}
d\mu_0 \right) \nonumber \\
=\ &
\frac{1}{\bar\rho} \sum_{k \geq 0} 
\left( \int_{\{ \varphi = \tau_{k+1}\}} 1_{\{j \geq \tau > j-\tau_k\}}  \circ f^{\tau_k}
d\mu_0 -
\int_{\{ \varphi > \tau_{k+1} \}} 1_{\{\tau > j\}}  \circ f^{\tau_k}
d\mu_0 \right),
\end{align*}
proving~\eqref{eq:mutau1}.
\end{proof}

We have the following corollary in the finite measure setting, for which we recall
that
$\bar \varphi = \int_Y \varphi \, d\mu_0$
and  $\bar\tau = \int_Y \tau \, d\mu_\tau$ are finite.

\begin{cor}\label{cor:mutau}
Suppose that $\mu_0$ and $\mu_\tau$ are equivalent
measures on $Y$ that are preserved by the general return
map $f^\varphi$ and first return map $f^\tau$, respectively.
Then 
$$
\left| \sum_{j \geq n} \frac{1}{\bar\varphi} \mu_0(\varphi > j) -
\frac{1}{\bar\tau} \mu_\tau(\tau > j) \right| \leq 
 \frac{1}{\bar\varphi} \sum_{k \geq 1} \int_{\{ \varphi = \tau_{k+1} > n\}}
\tau_k\, d\mu_0.
$$
\end{cor}

\begin{remark}\label{rem:mutau} In the setting of
\eqref{eq:f-neutral}, we can replace $k+1$ with $\tau_k$ in~\eqref{eq:k}
and obtain that $\sum_{k \geq 1} \int_{\{ \varphi = \tau_{k+1} > n\}}
\tau_k\, d\mu_0=O(\mu_0(\varphi > n))$.
This together with Corollary~\ref{cor:mutau}
implies that the leading term in Theorem~\ref{prop-corel-finite} applied to \eqref{eq:f-neutral}
matches the leading term of the correlation decay
results in~\cite{HuVaienti, MT}; although not exactly the same, the difference in the main terms
can be absorbed in the error term.
\end{remark}

\begin{proof}
Observe that
\begin{eqnarray*}\label{eq:varphitaurho}
\bar\varphi &=& \int_Y \varphi \, d\mu_0  = \mu_{\Delta}(\Delta)
= \sum_{k=0}^\infty \sum_{j=0}^{\tau_{k+1}-\tau_k-1} \mu_{\Delta}(Y_j)
= \sum_{k=0}^\infty \int_{\hat Y_k} \tau \circ f^{\tau_k} \, d\mu_{\Delta} 
\nonumber \\
&=& \int_Y \tau \, d\mu_{\Delta} \circ \pi^{-1} = \bar\rho \int_Y \tau \, d\mu_\tau
= \bar\rho \cdot \bar\tau.
\end{eqnarray*}
Therefore the statement of the corollary is equivalent to
\begin{equation}\label{eq:H1b}
\left| \sum_{j \geq n} \frac{1}{\bar\rho} \mu_0(\varphi > j) -
\mu_\tau(\tau > j) \right| \leq 
 \frac{1}{\bar\rho} \sum_{k \geq 1} \int_{\{ \varphi = \tau_{k+1} > n\}}
\tau_k\, d\mu_0.
\end{equation}
Continuing from Lemma~\ref{lem:mutau}, and since $\tau$ is constant
on $f^{\tau_k}(\{\varphi > \tau_{k+1}\})$, there is at most one $j \geq n$
for which $ 1_{\{\tau > j\}}  \circ f^{\tau_k} = 1$.
Therefore, using $\tau_{k+1} = \tau \circ f^{\tau_k} + \tau_k$, 
we can sum the second sum of the integrals in \eqref{eq:mutau1} over
$j \geq n$ and compute
\begin{align*}
\sum_{j \geq n} \sum_{k \geq 0} &
\int_{\{ \varphi > \tau_{k+1} \}} 1_{\{\tau > j\}}  \circ f^{\tau_k}\, d\mu_0
\leq \sum_{k \geq 0} \mu_0(\varphi > \tau_{k+1} > \tau_k + n) \\
\le& \sum_{k \geq 1} \mu_0(\varphi \geq \tau_{k+1} > n) 
\leq \sum_{k \geq 1} k \mu_0(\varphi = \tau_{k+1} > n)
= \sum_{k \geq 1} \int_{\{ \varphi = \tau_{k+1} > n\}} k \, d\mu_0,
\end{align*}
which is definitely less than the first sum of integrals in \eqref{eq:mutau1}, which we will
estimate now.

For the first sum of integrals in \eqref{eq:mutau1}, 
there are at most $\tau_k$ values of
$j \geq n$ making the indicator function $1$. 
Using again that $\tau_{k+1} = \tau \circ f^{\tau_k} + \tau_k$, 
we can sum the first sum of integrals in \eqref{eq:mutau1} over
$j \geq n$ and compute
\begin{eqnarray*}
\sum_{j \geq n} \sum_{k \geq 0}
\int_{\{ \varphi = \tau_{k+1}\}} 1_{\{j \geq \tau > j-\tau_k\}}  \circ f^{\tau_k}\,
d\mu_0
& \le& \sum_{k \geq 0} 
\int_{\{ \varphi = \tau_{k+1}\}}\tau_k\, 1_{\{\tau > n-\tau_k\}}  \circ f^{\tau_k}
 \, d\mu_0 \\
& \le& \sum_{k \geq 1} \int_{\{ \varphi = \tau_{k+1} > n \}}\tau_k \, d\mu_0.
\end{eqnarray*}
This proves \eqref{eq:H1b} and hence the corollary.
\end{proof}

\end{document}